\documentclass[a4paper]{amsart}

\usepackage{amsmath,amssymb,amsthm}
\usepackage[utf8]{inputenc}
\usepackage{scalerel}
\usepackage{multicol}
\usepackage{enumerate}
\usepackage{comment} 
\usepackage{microtype}
\usepackage[dvipsnames,svgnames,x11names,hyperref]{xcolor}
\definecolor{darkred}{RGB}{139,0,0}
\definecolor{darkblue}{RGB}{0,0,139}
\definecolor{darkgreen}{RGB}{0,100,0}
\usepackage[linktocpage]{hyperref}
\hypersetup{
  colorlinks   = true, 
  urlcolor     = darkred, 
  linkcolor    = darkred, 
  citecolor   = darkgreen}
\usepackage{tikz}
\usepackage{tikz-cd}
\usepackage{bbm}
\usepackage{verbatim}

\newtheorem{thm}{Theorem}[section]
\newtheorem{cor}[thm]{Corollary}
\newtheorem{lem}[thm]{Lemma}
\newtheorem{prop}[thm]{Proposition}

\theoremstyle{definition}
\newtheorem{defn}[thm]{Definition}
\theoremstyle{remark}
\newtheorem{rem}[thm]{Remark}
\newtheorem{example}[thm]{Example}
\numberwithin{equation}{section}

\newcommand{\bF}{\mathbb{F}}

\newcommand{\bL}{\mathbb{L}}

\newcommand{\bN}{\mathbb{N}}

\newcommand{\bQ}{\mathbb{Q}}
\newcommand{\bR}{\mathbb{R}}

\newcommand{\bZ}{\mathbb{Z}}

\newcommand{\gA}{\bold{A}}

\newcommand{\gC}{\bold{C}}

\newcommand{\gE}{\bold{E}}

\newcommand{\gM}{\bold{M}}

\newcommand{\gR}{\bold{R}}
\newcommand{\gS}{\bold{S}}
\newcommand{\gT}{\bold{T}}

\newcommand{\gX}{\bold{X}}

\newcommand\lra{\longrightarrow}

\newcommand\hocolim{\operatorname*{hocolim}}
\newcommand\colim{\operatorname*{colim}}

\newcommand\inj{\text{inj}}
\newcommand{\bunit}{\mathbbm{1}}
\newcommand{\bk}{\mathbbm{k}}

\title[Classical homological stability from the point of view of cells]{Classical homological stability\\from the point of view of cells}
\author{Oscar Randal-Williams}
\email{o.randal-williams@dpmms.cam.ac.uk}
\address{Centre for Mathematical Sciences\\
Wilberforce Road\\
Cambridge CB3 0WB\\
UK}
\date{\today}



\begin{document}

\begin{abstract}
We explain how to interpret the complexes arising in the ``classical'' homology stability argument (e.g.\ in the framework of Randal-Williams--Wahl) in terms of higher algebra, which leads to a new proof of homological stability in this setting. The key ingredient is a theorem of Damiolini on the contractibility of certain arc complexes. We also explain how to directly compare the connectivities of these complexes with that of the ``splitting complexes'' of Galatius--Kupers--Randal-Williams.
\end{abstract}

\maketitle

\section{Introduction}

The goal of this note is to compare the classical approach to homological stability, specifically the formalisation of Quillen's approach given by Wahl and myself \cite{RWW}, with the more recent approach via cellular $E_k$-algebras developed by Galatius, Kupers, and myself \cite{e2cellsI}. It is an insight of Krannich \cite{Krannich} that the proper generality for the classical approach is to work in the category of $\bN$-graded topological spaces, and start with a right $E_1$-module $\mathbf{M}$ over an $E_2$-algebra $\mathbf{R}$ equipped with compatible $\bN$-gradings, a stabilising element $\sigma \in \mathbf{R}(1)$, and then ask about homological stability of the sequence of maps
$$\mathbf{M}(0) \overset{- \cdot \sigma}\lra \mathbf{M}(1) \overset{- \cdot \sigma}\lra \mathbf{M}(2) \overset{- \cdot \sigma}\lra \mathbf{M}(3) \overset{- \cdot \sigma}\lra \cdots.$$
In practice one may often take $\mathbf{M} = \mathbf{R}$ with its right $\mathbf{R}$-action, but it is clarifying to separate the two notions: it is then clear \cite[Remark 2.19]{Krannich} that one may as well replace $\mathbf{R}$ by $\mathbf{E}_2^+(1_*(*))$, the free unital $E_2$-algebra on a single point in grading 1, and just consider the induced $\mathbf{E}_2^+(1_*(*))$-module structure on $\mathbf{M}$.

Viewed in this way, the constructions and results of \cite{RWW, Krannich} beg to be explained from the point of view of an $\mathbf{E}_2^+(1_*(*))$-module cell-structure on $\mathbf{M}$. Our first 
main result does this: in Theorem \ref{thm:1} we will show that the cofibre of Krannich's \cite[\S 2.2]{Krannich} ``canonical resolution'' $|R_\bullet(\mathbf{M})| \to \mathbf{M}$ may be identified with the derived $\mathbf{E}_2^+(1_*(*))$-module indecomposables of $\gM$, so that the high-connectivity of the ``spaces of destabilisations" $|W_\bullet(A)|$ implies a vanishing line for the $\mathbf{E}_2^+(1_*(*))$-module cells of $\mathbf{M}$ (at least after linearising). This leads to a new proof that the high-connectivity of the $|W_\bullet(A)|$ implies homological stability, which we explain in Section \ref{sec:StabRevis}. It also has consequences for homology with twisted coefficients, and for representation stability.

Our second main result is particular to the set up of \cite{RWW}, where a braided monoidal groupoid $\mathsf{G}$ (satisfying certain axioms) yields an $E_2$-algebra $\gR \simeq B\mathsf{G}$. In this setting, for a fixed stabilising object $\sigma$ of $\mathsf{G}$ and each object $A$ of $\mathsf{G}$ there is the space $|W_\bullet(A)|$ of destabilisations of $A$, as well as spaces $|Z^{E_1}_{\bullet}(A)|$ and $|Z^{E_2}_{\bullet, \bullet}(A)|$ of ``$E_1$- and $E_2$-splittings of $A$''. Proposition \ref{prop:TransfConn} will show that under appropriate conditions the homological connectivities of these three spaces are essentially equivalent.

\vspace{2ex}

\noindent\textbf{Acknowledgements.} I would like to thank M.\ Krannich and A.\ Kupers for feedback on an earlier draft of this paper, and the referee for their perspicacious comments. I was supported by the ERC under the European Union's Horizon 2020 research and innovation programme (grant agreement No.\ 756444) and by a Philip Leverhulme Prize from the Leverhulme Trust.

\section{Recollections}

There is some tension in comparing \cite{Krannich} and \cite{e2cellsI}, because although they both deal with $E_2$-algebras and modules over them these notions are implemented in technically different ways. Namely, in \cite{Krannich} Krannich considers a 2-coloured operad $\mathcal{O}$ 
equivalent to a certain suboperad the Swiss cheese operad $\mathcal{SC}_2$, whose algebras $(\mathbf{M}, \mathbf{R})$ are then considered as an $E_1$-module $\mathbf{M}$ over an $E_2$-algebra $\mathbf{R}$. On the other hand \cite{e2cellsI} considers unital algebras $\mathbf{R}^+$ over the little 2-cubes operad, constructs a strictification $\overline{\mathbf{R}}$ of the underlying $E_1$-algebra to an associative monoid, and then considers modules $\mathbf{M}$ over this monoid. While Krannich's formulation is more elegant, to take advantage of the large amount of machinery already developed in \cite{e2cellsI} we find it necessary to work in that setting, and in Section \ref{sec:strictifying} we will redevelop Krannich's ideas in that setting.

As the results we explain are principally of interest in the context of \cite{e2cellsI} we will freely use the basic notation and concepts of that paper without introducing them again, and only remind the reader of the most pressing or elaborate notions.

\subsection{$\bN$-graded spaces}

We shall often work in the category $\mathsf{Top}^\bN = \mathsf{Fun}(\bN, \mathsf{Top})$ of $\bN$-graded (compactly generated weak Hausdorff) topological spaces, where $\bN$ is considered as a category with only identity morphisms. An object $X$ of this category simply consists of a collection $\{X(n)\}_{n \in \bN}$ of spaces. If $V$ is a space, we write $n_*(V)$ for the $\bN$-graded spaces which is $V$ in grading $n$ and empty otherwise.

We endow this category with the symmetric monoidal structure $\otimes$ given by Day convolution, with $\bN$ considered as a symmetric monoidal category under addition. More prosaically, it is given by
$$(X \otimes Y)(n) = \coprod_{a+b=n} X(a) \times Y(n).$$
Using $\otimes$ we can therefore talk of associative algebra objects in $\mathsf{Top}^\bN$, and of modules over them. We may also talk of $E_k$-algebras in this category, as explained in \cite{e2cellsI}.

\subsection{The associative algebra $\gS$}\label{sec:R}

Following \cite[Section 12.2.1]{e2cellsI} we use the following model for $\overline{\gE}_2(1_*(*))$, an associative unital monoid equivalent as an $E_1$-algebra to the free unital $E_2$-algebra on one generator in grading 1. Let $\mathcal{C}_2(n)$ denote the $n$th space in the little 2-cubes operad, i.e.\ the space of tuples $e_1, e_2, \ldots, e_n : I^2 \to I^2$ of rectilinear embeddings having disjoint interiors.

\begin{defn}
Let $\gS$ be the $\bN$-graded space with
$$\gS(n) = (0,\infty) \times \mathcal{C}_2(n)/\Sigma_n$$
for $n> 0$, and $\gS(0)$ given by a single point, considered as $\{(0, \emptyset)\}$. We think of $\gS(n)$ as the space of pairs of a $t > 0$ and a set of $n$ unordered rectilinear embeddings $I^2 \to [0,t] \times [0,1]$ with disjoint interiors, by the evident rescaling. In this interpretation, translation and disjoint union provide maps
$$\gS(n) \times \gS(m) \lra \gS(n+m)$$
making $\gS$ into an associative unital monoid in $\bN$-graded spaces, with unit $(0, \emptyset)$. We write $\sigma := (1, \mathrm{id}_{I^2}) \in \gS(1)$, or equivalently $\sigma : 1_*(*) \to \gS$.
\end{defn}

There is a homotopy equivalence between $\gS(n)$ and the space $C_n(\bR^2)$ of configurations of $n$ unordered points in the plane (by passing first to the subspace with $t=1$, then considering the map which sends a collection of embeddings $\{e_i : I^2 \to [0,1]^2\}$ to the collection of their centres $\{e_i(\tfrac{1}{2}, \tfrac{1}{2}) \in (0,1)^2 \cong \bR^2\}$, which is a fibration with contractible fibres). As such, $\gS(n)$ is a model for the classifying space of Artin's braid group $\beta_n$ on $n$ strands. The map $- \cdot \sigma : \gS(n-1) \to \gS(n)$ corresponds to the homomorphism $\beta_{n-1} \to \beta_n$ which adds one strand (to the right). We record two well-known facts about these maps:
\begin{enumerate}[(i)]
\item\label{it:BraidInj} The homomorphism $\beta_{n-1} \to \beta_n$ is injective for all $n$.

\item\label{it:stab} The homomorphism $\beta_{n-1} \to \beta_n$ induces an isomorphism on homology in degrees $* \leq \tfrac{n-3}{2}$, and an epimorphism in degrees $* \leq \tfrac{n-1}{2}$. Equivalently, the relative homology groups satisfy $H_*(\beta_n, \beta_{n-1} ; \bZ)=0$ for $* < \tfrac{n}{2}$.
\end{enumerate}
The latter was first proved by Arnold \cite{Arnold}, and there are many more recent proofs. The former is easy: the homomorphism lands in the subgroup $\beta_{n-1,1} \leq \beta_n$ of those braids where the strand that starts at the rightmost point also ends at the rightmost point, and on this subgroup there is a splitting $\beta_{n-1,1} \to \beta_{n-1}$ given by forgetting this rightmost strand.

\subsection{Strictifying Krannich's framework}\label{sec:strictifying}

Let $\gM$ be a right $\gS$-module, and let us define the analogue of Krannich's ``canonical resolution'' \cite[\S 2.2]{Krannich}. This is almost a semi-simplicial ($\bN$-graded) space augmented over $\gM$, but is indexed on a topological category $\widetilde{\Delta}_{\inj}$ homotopy equivalent to, but not equal to, $\Delta_{\inj}$. We must first describe this category. We will work with Moore paths, write $\omega(\gamma)$ for the end of a Moore path $\gamma$, and write $*$ for concatenation of Moore paths.

\begin{defn}
For $[q], [p] \in \Delta_{\inj}$ let $U([q], [p])$ denote the space of pairs of a $d \in \gS(p-q)$ and a Moore path $\mu$ in $\gS(p+1)$ from $\sigma^{p+1}$ to $d \cdot \sigma^{q+1}$. There is a composition-law
$$U([l], [q]) \times U([q], [p]) \lra U([l], [p])$$
given by $((e, \gamma), (d, \mu)) \mapsto (d \cdot e, \mu * (d \cdot \gamma))$, giving the structure of a topologically enriched category $U$ with the same objects as $\Delta_\inj$.

For a morphism $i : [q] \to [p] \in \Delta_{\inj}$ there is a path component $U([q], [p])$ containing the point given by $d=\sigma^{p-q}$ and $\mu$ a Moore loop corresponding to the braid on $(p+1)$ strands where the first $q+1$ strands go behind the rest to end at $i([q]) \subset [p]$, as in Figure \ref{fig:1}. We let $\widetilde{\Delta}_{\inj}([q], [p]) \subset U([q], [p])$ consist of such path components; one checks it defines a subcategory of $U$ with the same objects as $\Delta_\inj$.
\end{defn}

\begin{figure}[h]
\tikzset{every picture/.style={line width=0.75pt}} 

\begin{tikzpicture}[x=0.5pt,y=0.5pt,yscale=-1,xscale=1]

\draw   (278.8,12) -- (570,12) -- (445.2,70.25) -- (154,70.25) -- cycle ;
\draw   (278.8,149) -- (570,149) -- (445.2,209.75) -- (154,209.75) -- cycle ;
\draw  [fill={rgb, 255:red, 0; green, 0; blue, 0 }  ,fill opacity=1 ] (275.5,40.5) .. controls (275.5,37.74) and (277.74,35.5) .. (280.5,35.5) .. controls (283.26,35.5) and (285.5,37.74) .. (285.5,40.5) .. controls (285.5,43.26) and (283.26,45.5) .. (280.5,45.5) .. controls (277.74,45.5) and (275.5,43.26) .. (275.5,40.5) -- cycle ;
\draw  [fill={rgb, 255:red, 0; green, 0; blue, 0 }  ,fill opacity=1 ] (354.5,40.5) .. controls (354.5,37.74) and (356.74,35.5) .. (359.5,35.5) .. controls (362.26,35.5) and (364.5,37.74) .. (364.5,40.5) .. controls (364.5,43.26) and (362.26,45.5) .. (359.5,45.5) .. controls (356.74,45.5) and (354.5,43.26) .. (354.5,40.5) -- cycle ;
\draw  [fill={rgb, 255:red, 0; green, 0; blue, 0 }  ,fill opacity=1 ] (314.5,40.5) .. controls (314.5,37.74) and (316.74,35.5) .. (319.5,35.5) .. controls (322.26,35.5) and (324.5,37.74) .. (324.5,40.5) .. controls (324.5,43.26) and (322.26,45.5) .. (319.5,45.5) .. controls (316.74,45.5) and (314.5,43.26) .. (314.5,40.5) -- cycle ;
\draw  [fill={rgb, 255:red, 0; green, 0; blue, 0 }  ,fill opacity=1 ] (435.5,40.5) .. controls (435.5,37.74) and (437.74,35.5) .. (440.5,35.5) .. controls (443.26,35.5) and (445.5,37.74) .. (445.5,40.5) .. controls (445.5,43.26) and (443.26,45.5) .. (440.5,45.5) .. controls (437.74,45.5) and (435.5,43.26) .. (435.5,40.5) -- cycle ;
\draw  [fill={rgb, 255:red, 4; green, 0; blue, 0 }  ,fill opacity=1 ] (395.5,40.5) .. controls (395.5,37.74) and (397.74,35.5) .. (400.5,35.5) .. controls (403.26,35.5) and (405.5,37.74) .. (405.5,40.5) .. controls (405.5,43.26) and (403.26,45.5) .. (400.5,45.5) .. controls (397.74,45.5) and (395.5,43.26) .. (395.5,40.5) -- cycle ;
\draw  [fill={rgb, 255:red, 0; green, 0; blue, 0 }  ,fill opacity=1 ] (275,180.5) .. controls (275,177.74) and (277.24,175.5) .. (280,175.5) .. controls (282.76,175.5) and (285,177.74) .. (285,180.5) .. controls (285,183.26) and (282.76,185.5) .. (280,185.5) .. controls (277.24,185.5) and (275,183.26) .. (275,180.5) -- cycle ;
\draw   (315,180.75) .. controls (315,177.99) and (317.24,175.75) .. (320,175.75) .. controls (322.76,175.75) and (325,177.99) .. (325,180.75) .. controls (325,183.51) and (322.76,185.75) .. (320,185.75) .. controls (317.24,185.75) and (315,183.51) .. (315,180.75) -- cycle ;
\draw  [fill={rgb, 255:red, 0; green, 0; blue, 0 }  ,fill opacity=1 ] (355.43,179.75) .. controls (355.43,176.99) and (357.67,174.75) .. (360.43,174.75) .. controls (363.19,174.75) and (365.43,176.99) .. (365.43,179.75) .. controls (365.43,182.51) and (363.19,184.75) .. (360.43,184.75) .. controls (357.67,184.75) and (355.43,182.51) .. (355.43,179.75) -- cycle ;
\draw  [fill={rgb, 255:red, 0; green, 0; blue, 0 }  ,fill opacity=1 ] (395.5,179.75) .. controls (395.5,176.99) and (397.74,174.75) .. (400.5,174.75) .. controls (403.26,174.75) and (405.5,176.99) .. (405.5,179.75) .. controls (405.5,182.51) and (403.26,184.75) .. (400.5,184.75) .. controls (397.74,184.75) and (395.5,182.51) .. (395.5,179.75) -- cycle ;
\draw   (435.5,179.75) .. controls (435.5,176.99) and (437.74,174.75) .. (440.5,174.75) .. controls (443.26,174.75) and (445.5,176.99) .. (445.5,179.75) .. controls (445.5,182.51) and (443.26,184.75) .. (440.5,184.75) .. controls (437.74,184.75) and (435.5,182.51) .. (435.5,179.75) -- cycle ;
\draw    (319.5,40.5) .. controls (320.14,69.57) and (333,73.29) .. (341.29,78.71) ;
\draw    (440.5,40.5) .. controls (440,83.25) and (400,142.25) .. (400.5,179.75) ;
\draw    (280.5,40.5) .. controls (280.43,70.14) and (292.19,97.97) .. (303.86,125.57) ;
\draw    (359.5,40.5) .. controls (360.5,84.25) and (279.5,142.75) .. (280,180.5) ;
\draw    (360.43,180.75) .. controls (359.5,143.25) and (400,83.25) .. (400.5,40.5) ;
\draw    (308.43,133.86) .. controls (313.71,144.57) and (320.14,149.57) .. (320,175.75) ;
\draw    (347.14,82) .. controls (350.29,85.14) and (363.29,93.29) .. (381.29,102.14) ;
\draw    (421.71,116.71) .. controls (434.71,121.57) and (440.14,148.43) .. (439.5,174.75) ;
\draw    (385.86,104.14) .. controls (389.57,105.86) and (409.57,112.14) .. (414.86,114.57) ;

\end{tikzpicture}

	\caption{The braid $\mu$, where the points $i([q])$ are shown open.}
	\label{fig:1}
\end{figure}
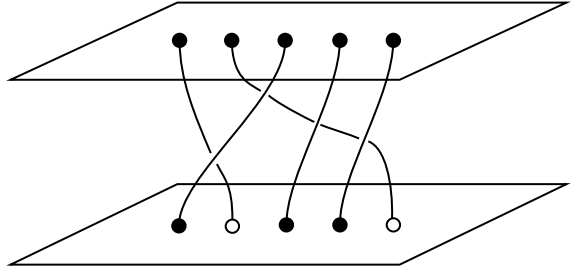

As in \cite[Lemma 2.11]{Krannich} the space $\widetilde{\Delta}_{\inj}([q], [p])$ is homotopy discrete, and the map
$${\Delta}_{\inj}([q], [p]) \lra \pi_0\widetilde{\Delta}_{\inj}([q], [p])$$
described above is a bijection. This yields a functor $\widetilde{\Delta}_{\inj} \to {\Delta}_{\inj}$ which is the identity on objects, and an equivalence on morphism spaces.

\begin{defn}
For a right $\gS$-module $\gM$, let $R_p(\gM)$ denote the $\bN$-graded space which in grading $n$ consists of pairs of a point $a \in \gM(n-p-1)$ and a Moore path $\gamma$ in $\gM(n)$ ending at $a \cdot \sigma^{p+1}$. Consider this as an enriched functor $\widetilde{\Delta}_{\inj}^{op} \to \mathsf{Top}$ via the maps
$$\widetilde{\Delta}_{\inj}([q], [p]) \times R_p(\gM) \lra R_q(\gM)$$
given by $((d, \mu), (a, \gamma)) \mapsto (a \cdot d, \gamma * (a \cdot \mu))$. Evaluating the Moore path at 0 gives an augmentation $R_\bullet(\gM) \to \gM$.
\end{defn}

One then sets
$$|R_\bullet(\gM)| := \hocolim_{[p] \in \widetilde{\Delta}_{\inj}^{op}} R_p(\gM)$$
by analogy with the geometric realisation of a semi-simplicial space. Krannich's development of homological stability in this setting takes as its axiom the high-connectivity of the map $\epsilon_\gM : |R_\bullet(\gM)| \to \gM$, specifically that there is a $k \geq 2$ such that this map is $\lfloor \tfrac{n-2+k}{k} \rfloor$-connected in grading $n$, for all $n$.

\section{The canonical resolution and module indecomposables}\label{sec:CanRes}

Our main result is the following, showing that the homotopy cofibre of the ``canonical resolution'' has a conceptual meaning: it is the derived $\gS$-module indecomposables. 

\begin{thm}\label{thm:1}
If $\gM$ is a right $\gS$-module, there is an equivalence of $\bN$-graded spaces between the homotopy cofibre of $\epsilon_\gM : |R_\bullet(\gM)| \to \gM$ and $Q_\bL^{\gS}(\gM)$.
\end{thm}
Let us write $\gS_{>0}$ for the sub-$\bN$-graded space of $\gS$ which is empty in grading 0 and agrees with $\gS$ otherwise. The following expands upon \cite[Example 2.18]{Krannich}.

\begin{lem}\label{lem:Damiolini}
The augmentation $\epsilon_\gS : |R_\bullet(\gS)| \to \gS$ is an equivalence onto $\gS_{>0}$.
\end{lem}
\begin{proof}
The space $R_0(\gS)(0)$ consists of a point $a \in \gS(-1)$ and a Moore path to $a \cdot \sigma$, so is empty, and so the fibre of $\epsilon_\gS$ over the point of grading 0 is indeed empty.

The homotopy fibre of $\epsilon_\gS$ over $b \in \gS(n)$ with $n>0$ is, using \cite[p.\ 180]{Farjoun}, equivalent to the realisation of the $\widetilde{\Delta}_{\inj}$-space given by the homotopy fibres over $b$ of the maps $R_p(\gS) \to \gS$, but as these maps are fibrations this is in turn the same as the realisation of the $\widetilde{\Delta}_{\inj}$-space $F_\bullet(b)$ with $p$-simplices given by the literal fibres of these maps, i.e.\ an $a \in \gS(n-p-1)$ and a Moore path from $b$ to $a \cdot \sigma^{p+1}$. There are maps
$$\hocolim_{[p] \in \widetilde{\Delta}_{\inj}^{op}} F_p(b) \lra \hocolim_{[p] \in \widetilde{\Delta}_{\inj}^{op}} \pi_0 F_p(b) \lra \hocolim_{[p] \in {\Delta}_{\inj}^{op}} \pi_0 F_p(b)$$
induced by $F_\bullet(b) \to \pi_0 F_\bullet(b)$, and by the fact that the functor $\pi_0 F_\bullet(b)$ on $\widetilde{\Delta}_{\inj}^{op}$ canonically factors through $\widetilde{\Delta}_{\inj}^{op} \to {\Delta}_{\inj}^{op}$. The second map is an equivalence as $\widetilde{\Delta}_{\inj}^{op} \to {\Delta}_{\inj}^{op}$ is an equivalence of enriched categories. The first map is an equivalence as each $F_p(b)$ is homotopy-discrete: this is because it is a homotopy fibre of the map $\gS(n-p-1) \to \gS(n)$, and by item (\ref{it:BraidInj}) of Section \ref{sec:R} this is a map of $K(\pi, 1)$'s which is injective on fundamental groups.

We are left with needing to show that the semi-simplicial set $[p] \mapsto \pi_0 F_p(b)$ has contractible geometric realisation. For any $b$ this is the ``space of destabilisations'' of \cite[Definition 2.1]{RWW} in the case of the braid groups, which is described in \cite[Section 5.6.2]{RWW} as an arc complex. By a remarkable theorem of Damiolini \cite[Theorem 2.48]{Damiolini} (see \cite[Proposition 3.2]{HatcherVogtmann} for a published reference) this arc complex is contractible.
\end{proof}

\begin{proof}[Proof of Theorem \ref{thm:1}]
The augmented $\widetilde{\Delta}_{\inj}$-space $R_\bullet(\gS) \to \gS$ is constructed using the right $\gS$-module structure on $\gS$, so it admits a compatible left $\gS$-module structure
via
$$(b, (a, \gamma)) \longmapsto (b\cdot a, b \cdot \gamma): \gS \otimes R_p(\gS) \lra R_p(\gS).$$
Furthermore, contracting the Moore path gives a deformation retraction from $R_p(\gS)$ to the subspace where the Moore path is trivial, and this subspace is isomorphic to $\gS \otimes (p+1)_*(*)$ as a $\bN$-graded space, and as a left $\gS$-module. 

There is a map
$$\phi_p' : \gM \otimes R_p(\gS) \lra R_p(\gM)$$
given by $(a, (b, \gamma)) \mapsto (a \cdot b, a \cdot \gamma)$. If $c \in \gS$ then the map above satisfies $\phi_p'(a \cdot c, b) = \phi_p'(a, c \cdot b)$, and hence descends to a map $\phi_p : \gM \otimes_\gS R_p(\gS) \to R_p(\gM)$ from the coequaliser. The composition
$$B(\gM, \gS, R_p(\gS)) \lra \gM \otimes_\gS R_p(\gS) \overset{\phi_p}\lra R_p(\gM)$$
of the augmentation map and $\phi_p$ is an equivalence, using $R_p(\gM) \simeq \gM\otimes (p+1)_*(*)$ as well as $R_p(\gS) \simeq \gS\otimes (p+1)_*(*)$ and $B(\gM, \gS, \gS) \simeq \gM$. By commuting the bar construction with the homotopy colimit defining $|R_\bullet(\gS)|$, and using Lemma \ref{lem:Damiolini}, we obtain equivalences
$$B(\gM, \gS, \gS_{>0}) \overset{\sim}\longleftarrow B(\gM, \gS, |R_\bullet(\gS)|) \overset{\sim}\lra |R_\bullet(\gM)|$$
over $\gM$. This identifies the homotopy cofibre of $\epsilon_\gM$ with the homotopy cofibre of the composition
$$B(\gM, \gS, \gS_{>0}) \lra B(\gM, \gS, \gS) \overset{\sim}\lra \gM,$$
which is equivalent to the homotopy cofibre of the first map, i.e.\ $B(\gM, \gS, \gS/\gS_{>0}) \simeq B(\gM, \gS, \bunit)$. The latter bar construction should be interpreted as being formed in $\bN$-graded pointed spaces, with $\gM$ and $\gS$ included in this category by implicitly adding disjoint basepoints, and with $\bunit \simeq \gS/\gS_{>0}$ given by
$$\bunit(n) = \begin{cases}
S^0 & n=0\\
* & n>0.
\end{cases}$$
By \cite[Corollary 9.17]{e2cellsI} $B(\gM, \gS, \bunit)$ agrees with $Q_\bL^{\gS}(\gM)$, as long as $\gS$ and $\gM$ are cofibrant in $\mathsf{Top}^\bN$.

Now each $\gS(n)$ has the structure of a smooth manifold with corners, so is cofibrant in $\mathsf{Top}$; on the other hand the cofibrancy hypotheses in $\gM$ may be neglected, for the following reason. If $\gM^c \overset{\sim}\to \gM$ is a cofibrant replacement of $\gM$ as a $\gS$-module (and so in particular $\gM^c$ is cofibrant in $\mathsf{Top}$) then the above applies to give $B(\gM^c, \gS, \bunit) \simeq Q^{\gS}_\bL(\gM^c)$. Now certainly $Q^{\gS}_\bL(\gM^c) \to Q^{\gS}_\bL(\gM)$ is an equivalence, but also $B_\bullet(\gM^c, \gS, \bunit) \to B_\bullet(\gM, \gS, \bunit)$ is a levelwise equivalence (cartesian product preserves equivalences between \emph{all} objects in $\mathsf{Top}$) and so $B(\gM^c, \gS, \bunit) \to B(\gM, \gS, \bunit)$ is an equivalence too (geometric realisation of semi-simplicial spaces preserves equivalences between \emph{all} semi-simplicial (compactly-generated) spaces \cite[Theorem 2.2]{ERWSx}).
\end{proof}

\section{Classical homological stability revisited}\label{sec:StabRevis}

Theorem \ref{thm:1} leads to a new proof of homological stability in the setting of \cite{Krannich} or \cite{RWW} (adapted as in Section \ref{sec:strictifying}), quite different from the standard proof but very similar in spirit to \cite[\S 18]{e2cellsI}. It takes as given homological stability (of slope $\tfrac{1}{2}$) for the free $E_2$-algebra on one generator, i.e.\ configuration spaces of little cubes (or points) in the plane, or equivalently the braid groups.

In the terms we have been using, homological stability may be formulated as follows. Firstly, if $\gX$ is an $\bN$-graded space then we define bigraded homology groups by $H_{n,d}(\gX) := H_d(\gX(n))$, and similarly reduced homology groups of $\bN$-graded pointed spaces. Secondly, if $\gM$ is a right $\gS$-module then we may form the composition
$$- \cdot \sigma : \gM \otimes 1_*(*) \overset{\gM \otimes \sigma}\lra \gM \otimes \gS \overset{\cdot}\lra \gM$$
in the category of $\bN$-graded spaces, and write $\gM/\sigma$ for its homotopy cofibre (considered as a $\bN$-graded pointed space). Using that $(\gM \otimes 1_*(*))(n) = \gM(n-1)$, the associated long exact sequence on homology takes the form
$$\cdots \lra \widetilde{H}_{n,d+1}(\gM/\sigma) \lra H_d(\gM(n-1)) \overset{(- \cdot\sigma)_*}\lra H_d(\gM(n)) \lra \widetilde{H}_{n,d}(\gM/\sigma) \lra \cdots.$$
Thus homological stability of the sequence of maps
\begin{equation}\label{eq:StabSeq}
\mathbf{M}(0) \overset{- \cdot \sigma}\lra \mathbf{M}(1) \overset{- \cdot \sigma}\lra \mathbf{M}(2) \overset{- \cdot \sigma}\lra \mathbf{M}(3) \overset{- \cdot \sigma}\lra \cdots
\end{equation}
corresponds to the vanishing of the groups $\widetilde{H}_{n,d}(\gM/\sigma)$ for $d \ll n$.

Henceforth $\bk$ will always denote a commutative ring.

\begin{prop}\label{prop:homstab}
Let $\gM$ be a right $\gS$-module and $f : \bN \to \bN$ be such that
$$H_{n,d}(\gM, |R_\bullet(\gM)| ; \bk )=0 \text{ for } d < f(n).$$
Then, setting $\bar{f}(n) := \min \{\lfloor f(p) + \tfrac{1}{2}(n-p)\rfloor \,|\, 0 \leq p \leq n\}$, we have
$$\widetilde{H}_{n,d}(\gM/\sigma ; \bk)=0 \text{ for } d < \bar{f}(n).$$
\end{prop}
In particular, if $f$ diverges then so does $\bar{f}$, i.e.\ \eqref{eq:StabSeq} satisfies homological stability.

\begin{proof}[Proof of Proposition \ref{prop:homstab}]
By Theorem \ref{thm:1}, the hypothesis of the proposition is equivalent to $H^{\gS}_{n,d}(\gM ; \bk)=0$ for $d < f(n)$. Applying the symmetric monoidal functor $(-)_\bk := \bk[\mathrm{Sing}_\bullet(-)] : \mathsf{Top} \to \mathsf{sMod}_\bk$ we obtain an associative monoid $\gS_\bk$ and a module $\gM_\bk$ over it in the category $\mathsf{sMod}_\bk^\bN$, satisfying $H^{\gS_\bk}_{n,d}(\gM_\bk)=0$ for $d < f(n)$. By \cite[Theorem 11.21]{e2cellsI} we may find a $\gS_\bk$-module cellular approximation $\gC \overset{\sim}\to \gM_\bZ$, such that $\gC$ only has $(n,d)$-cells with $d \geq f(n)$. We write $\mathrm{sk}(\gC)$ for the filtered $\gS_\bk$-module given by the skeletal filtration of $\gC$.

By considering the functor $(-)/\sigma = (-) \otimes_{\gS_\bk} \gS_\bk/\sigma$, which preserves homotopy cofibre sequences of right $\gS_\bk$-modules, we obtain a filtration of $\gC/\sigma$ with associated graded
$$\mathrm{gr}(\gC/\sigma) \simeq \mathrm{gr}(\gC)/\sigma \simeq \bigoplus_{d \geq 0} \bigoplus_{\alpha \in I_d} S^{n_\alpha, d} \otimes \gS_\bZ/\sigma,$$
where $d \geq f(n_\alpha)$ for $\alpha \in I_d$.

By the discussion in Section \ref{sec:R} we have $H_{n,d}(\gS_\bk/\sigma) \cong H_d(\beta_n, \beta_{n-1};\bk)$ for $\beta_n$ the $n$th braid group, and by item (\ref{it:stab}) of Section \ref{sec:R} and the Universal Coefficient Theorem this vanishes for $d < \tfrac{n}{2}$. It follows that the homology of $\mathrm{gr}(\gC/\sigma)$ vanishes in bidegrees $(n,d)$ such that $d < f(n_\alpha) + \tfrac{n-n_\alpha}{2}$ for all cells $\alpha$, so in particular for $d < \bar{f}(n)$. The same then holds for $\gC/\sigma \simeq \gS_\bZ/\sigma$ by the spectral sequence for the skeletal filtration of $\gC/\sigma$.
\end{proof}

This is a simple application of Theorem \ref{thm:1}, using only that $\gS$ enjoys homological stability of slope $\tfrac{1}{2}$ with integral coefficients. But the principle behind the argument above shows that $\gM$ will enjoy any homological stability pattern that $\gS$ does, in a range of degrees controlled by the vanishing of $H_{*,*}(\gM, |R_\bullet(\gM)|)$. (Of course this is only useful when the latter has a vanishing line of slope $> \tfrac{1}{2}$: Coxeter groups \cite[Section 8]{HepworthCoxeter} and Artin monoids \cite[Theorem 8.1]{BoydArtin} give good families of examples.) As the homology of $\gS$ is completely known, such patterns (meaning improved homological stability ranges with $\bQ$- or $\bF_p$-coefficients, or secondary and higher homological stability) can be easily analysed. A detailed analysis is given in \cite[Corollary 2.12]{Himes}: we will not spell out the (rather involved) formulation here.

A converse to Proposition \ref{prop:homstab} holds too: 

\begin{prop}\label{prop:homstabConverse}
Let $\gM$ be a right $\gS$-module, and $g : \bN \to \bN$ be such that
$$\widetilde{H}_{n,d}(\gM/\sigma ; \bk)=0 \text{ for } d < g(n).$$
Then, setting $\bar{g}(n) := \min \{g(p) + (n-p) \, | \, 0 \leq p \leq n\}$, we have
\[H_{n,d}(\gM, |R_\bullet(\gM)| ; \bk)=0 \text{ for } d < \bar{g}(n).\makeatletter\displaymath@qed\]
\end{prop}

In practice this is not usually sharp, in that $H_{*,*}(\gM, |R_\bullet(\gM)| ; \bk)$ often vanishes with larger slope than $\widetilde{H}_{*,*}(\gM/\sigma ; \bk)$ does. As mentioned above, this usually indicates the presence of secondary and higher order homological stability for $\gS$.

In view of Theorem \ref{thm:1}, a highbrow proof of Proposition \ref{prop:homstabConverse} is the discussion in \cite[Remark 19.3]{e2cellsI}, allowing oneself to be more flexible with the form of the stability ranges. A middlebrow proof is to consider the Bousfield--Kan spectral sequence for the augmented $\widetilde{\Delta}_{\inj}^{op}$-space $\epsilon_\gM : R_\bullet(\gM) \to \gM$, which---as the morphism spaces in $\widetilde{\Delta}_{\inj}^{op}$ are homotopy discrete---takes the form
$$E^1_{n,p,q} = H_{n,q}(R_p(\gM);\bk) \Longrightarrow H_{n, p+q+1}(\gM, |R_\bullet(\gM)| ; \bk)$$
for $p \geq -1$ with $R_{-1}(\gM) := \gM$. As $R_p(\gM) \simeq \gM \otimes (p+1)_*(*)$ we can write the $E^1$-page as $E^1_{n,p,q} \cong H_{q}(\gM(n-p-1);\bk)$, and recognise the $d^1$-differential $d^1 : E^1_{n,p,q} \to E^1_{n,p-1,q}$ as the alternating sum of $p+1$ copies of the stabilisation map $(- \cdot \sigma)_*$. Thus this differential is zero if $p$ is odd, and is $(-\cdot\sigma)_*$ if $p$ is even. From the assumption it is then easy to see that $E^2_{n,p,q}=0$ for $p+q+1 < \bar{g}(n)$. This is simply the usual spectral sequence argument for homological stability, with the logic reversed. (As with many middlebrow arguments, it even offers a slight improvement: in the definition of $\bar{g}$ one can take the minimum over those $0 \leq p \leq n$ having the same parity as $n$.)

\section{An extension of Theorem \ref{thm:1}}\label{sec:RepStab}

The discussion of Section \ref{sec:CanRes} shows that the cofibre of the canonical resolution $\epsilon_\gM : |R_\bullet(\gM)| \to \gM$ is equivalent to the derived $\gS$-module indecomposables $Q^\gS_\bL(\gM)$, so that the high-connectivity of this cofibre means that $\gM$ can be constructed as a cellular $\gS$-module without using small-dimensional $\gS$-module cells in large $\bN$-grading. Usually, such high-connectivity is proved by establishing the high-connectivity of the fibres of $\epsilon_\gM$: the fibre $W_\bullet(m)$ of $\epsilon_\gM : R_\bullet(\gM) \to \gM$ over a point $m \in \gM$ is called the ``space of destabilisations'' in \cite[Definition 2.14 (ii)]{Krannich}. The high-connectivity of a fibre is, of course, stronger than the high-connectivity of the corresponding cofibre. As such it might be expected that the high-connectivity of the fibres $W_\bullet(m)$ has more consequences than the high-connectivity of $Q^\gS_\bL(\gM)$. The goal of this section is to explain how this is so.

\subsection{Formulation}\label{sec:setup}
The theory in \cite{e2cellsI} is developed not only in the category $\mathsf{Top}^\bN$ of $\bN$-graded spaces, but more generally in $\mathsf{G}$-graded spaces for a (symmetric or braided) monoidal groupoid $\mathsf{G}$. This allows for the treatment of homological stability with twisted coefficients, and is also the natural context for representation stability. 

Let $(\mathsf{G}, \oplus, b, 0)$ be a braided monoidal groupoid. Let $r : \mathsf{G} \to \bN$ be a strong monoidal functor, called the \emph{rank}, and choose an $X \in \mathsf{G}$ with $r(X)=1$. Assume furthermore that 
\begin{enumerate}[(I)]
\item\label{it:NoRk0} $0 \in \mathsf{G}$ is the only object of rank 0, and

\item\label{it:TrivAut0} $\mathrm{Aut}_\mathsf{G}(0)$ is trivial.
\end{enumerate}
Endow $\mathsf{Top}^\mathsf{G} = \mathsf{Fun}(\mathsf{G}, \mathsf{Top})$ with the braided monoidal structure given by Day convolution, and similarly $\mathsf{sMod}^\mathsf{G}_\bZ$. 

In order to discuss $E_2$-algebras in a category which is only braided monoidal, in \cite[Section 4.1]{e2cellsI} there is introduced the category $\mathsf{FB}_2$ of ``braided finite sets'', and the category $\mathsf{Top}^{\mathsf{FB}_2}$ replaces the category of symmetric sequences. It is endowed with a composition product \cite[Definition 4.3]{e2cellsI}, monoids for which serve as a braided version of operads. In particular there is a braided version $\mathcal{C}^{\mathsf{FB}_2}_2$ of the non-unitary little 2-cubes operad \cite[Definition 12.6]{e2cellsI}. This has $\mathcal{C}^{\mathsf{FB}_2}_2(n)$ contractible for each $n>0$ (and empty for $n=0$).

Using this we can make sense of $E_2$-algebras in $\mathsf{Top}^\mathsf{G}$ or $\mathsf{sMod}^\mathsf{G}_\bZ$, and in particular we can form the free $E_2$-algebra on the object $X_*(*) \in \mathsf{Top}^\mathsf{G}$,
$$\mathbf{E}_2(X_*(*)) \in \mathsf{Alg}_{E_2}(\mathsf{Top}^\mathsf{G}).$$
We can strictify $\mathbf{E}_2(X_*(*))$ to a unital associative algebra 
$$\widetilde{\gS} := \overline{\mathbf{E}}_2(X_*(*)),$$
which plays the role of $\gS$ in this setting, and consider a right $\widetilde{\gS}$-module $\widetilde{\gM}$. The object $\widetilde{\gS}$ is cofibrant in $\mathsf{Top}^\mathsf{G}$, and we will always assume that $\widetilde{\gM}$ is too.

Taking left Kan extension along $r : \mathsf{G} \to \bN$ gives
$$r_* \widetilde{\gS} = r_* \overline{\gE}_2(X_*(*)) = \overline{\gE}_2(1_*(*)) = \gS \quad\text{ and }\quad r_* \widetilde{\gM} =: \gM,$$
(as these objects were cofibrant in $\mathsf{Top}^\mathsf{G}$, this agrees with the homotopy left Kan extension) and $\gM$ is a right $\gS$-module. For each object $Y \in \mathsf{G}$ there is a quotient map $q : \widetilde{\gM}(Y) \to \gM(r(Y))$. This puts us in the setting of Section \ref{sec:strictifying}: there is the canonical resolution $\epsilon_\gM : R_\bullet(\gM) \to \gM$, with fibre $W_\bullet(m)$ over $m \in \gM$.

The following relates the spaces $|W_\bullet(m)|$, in particular their connectivities, with the derived $\widetilde{\gS}$-module indecomposables.

\begin{thm}\label{thm:repstab}
Let $\widetilde{\gM}$ be a right $\widetilde{\gS}$-module which is cofibrant in $\mathsf{Top}^\mathsf{G}$. Then there is a morphism
\begin{equation}\label{eq:PreKanExt}
B(\widetilde{\gM}, \widetilde{\gS}, \widetilde{\gS}_{>0}) \lra \widetilde{\gM}.
\end{equation}
with homotopy cofibre $Q_\bL^{\widetilde{\gS}}(\widetilde{\gM})$ and homotopy fibre over $\tilde{m} \in \widetilde{\gM}$ given by $|W_\bullet(q(\tilde{m}))|$.

In particular, if $|W_\bullet(q(\tilde{m}))|$ is $k$-connected for all $\tilde{m} \in \widetilde{\gM}(Y)$, then $Q_\bL^{\widetilde{\gS}}(\widetilde{\gM})(Y)$ is $(k+1)$-connected.
\end{thm}

\subsection{Proof of Theorem \ref{thm:repstab}}\label{sec:Pfthm:repstab}

We prove this theorem by analogy with Theorem \ref{thm:1}, and so first construct an augmented $\widetilde{\Delta}_\inj^{op}$-object $R_\bullet(\widetilde{\gM}) \to \widetilde{\gM}$. For each object $Y \in \mathsf{G}$ and each $[p] \in \widetilde{\Delta}_\inj^{op}$ we define $R_p(\widetilde{\gM})(Y)$ by the cartesian square
\begin{equation}\label{eq:CartSq}
  \begin{tikzcd}
R_p(\widetilde{\gM})(Y) \dar \rar& \widetilde{\gM}(Y) \dar{q}\\
R_p(\gM)(r(Y)) \rar & \gM(r(Y)).
  \end{tikzcd}
\end{equation}
Repeatedly using the universal property of pullbacks, we see that these assemble to $R_p(\widetilde{\gM}) \in \mathsf{Top}^\mathsf{G}$, and that in turn these assemble to an augmented $\widetilde{\Delta}_\inj^{op}$-object $R_\bullet(\widetilde{\gM}) \to \widetilde{\gM}$ in $\mathsf{Top}^\mathsf{G}$. Furthermore, when $\widetilde{\gM} = \widetilde{\gS}$ we see that this object consists of left $\widetilde{\gS}$-modules.

As $\widetilde{\gM}$ is assumed to be cofibrant, the quotient map $\widetilde{\gM}(Y) \to \widetilde{\gM}(Y)/\mathrm{Aut}_\mathsf{G}(Y)$ is a covering space and the latter is a union of path-components of $\gM(r(Y))$, so the right-hand vertical map in \eqref{eq:CartSq} is a fibration: thus this square is also homotopy cartesian. It then follows that the square
\begin{equation}\label{eq:CartSq2}
  \begin{tikzcd}
{|R_\bullet(\widetilde{\gM})|(Y)} \dar \rar& \widetilde{\gM}(Y) \dar{q}\\
{|R_\bullet(\gM)|(r(Y))} \rar & \gM(r(Y))
  \end{tikzcd}
\end{equation}
is also homotopy cartesian (\cite[Lemma 2.13]{ERWSx} gives this for $\Delta_\inj^{op}$-objects: it follows for $\widetilde{\Delta}_\inj^{op}$-objects by first homotopy Kan extending along the equivalence of enriched categories $\widetilde{\Delta}_\inj^{op} \to \Delta_\inj^{op}$).

If we let $\widetilde{\gS}_{>0} \in \mathsf{Top}^\mathsf{G}$ be the object that agrees with $\widetilde{\gS}$ on objects $Y$ with $r(Y)>0$, and is the empty space on objects $Y$ with $r(Y)=0$ (recall that we have assumed that $0 \in \mathsf{G}$ is the only such object), then this obtains the structure of a left $\widetilde{\gS}$-module. It follows from Lemma \ref{lem:Damiolini} and the homotopy cartesian square \eqref{eq:CartSq2} that the augmentation gives an equivalence $|R_\bullet(\widetilde{\gS})| \to \widetilde{\gS}_{>0}$ of left $\widetilde{\gS}$-modules.

\begin{proof}[Proof of Theorem \ref{thm:repstab}]
We proceed as in the proof of Theorem \ref{thm:1}. Applying $B(\widetilde{\gM}, \widetilde{\gS}, -)$ to the homotopy cofibre sequence
$$\widetilde{\gS}_{>0} \lra \widetilde{\gS} \lra \bunit,$$
and using $B(\widetilde{\gM}, \widetilde{\gS}, \widetilde{\gS}) \overset{\sim}\to \widetilde{\gM}$, constructs the map \eqref{eq:PreKanExt} and identifies its homotopy cofibre with $B(\widetilde{\gM}, \widetilde{\gS}, \bunit )$, which is equivalent to $Q_\bL^{\widetilde{\gS}}(\widetilde{\gM})$ by \cite[Corollary 9.17]{e2cellsI}. 

On the other hand there are equivalences
$$B(\widetilde{\gM}, \widetilde{\gS}, \widetilde{\gS}_{>0}) \overset{\sim}\longleftarrow B(\widetilde{\gM}, \widetilde{\gS}, |R_\bullet(\widetilde{\gS})|) \overset{\sim}\lra |R_\bullet(\widetilde{\gM})|$$
over $\widetilde{\gM}$, using as in  the proof of Theorem \ref{thm:1} that $R_p(\widetilde{\gS}) \simeq \widetilde{\gS} \otimes (X^{\oplus p+1})_*(*)$ as a left $\widetilde{\gS}$-module, and similarly for $\widetilde{\gM}$. Finally, the homotopy fibre of $|R_\bullet(\widetilde{\gM})|(Y) \to \widetilde{\gM}(Y)$ over $\tilde{m} \in \widetilde{\gM}(Y)$ is $|W_\bullet(q(\tilde{m}))|$ as \eqref{eq:CartSq2} is homotopy cartesian.
\end{proof}

\subsection{$E_2$-algebras coming from groupoids}\label{sec:AlgFromGpd}

A useful application of this result is as follows. As in \cite[Section 17.1]{e2cellsI} (but replacing $\mathsf{sSet}$ by $\mathsf{Top}$) there is a $\gT \in \mathsf{Alg}_{E_2}(\mathsf{Top}^\mathsf{G})$ with $\gT(A) \simeq *$ if $r(A)>0$ and $\gT(0)=\emptyset$, which is also cofibrant in $\mathsf{Alg}_{E_2}(\mathsf{Top}^\mathsf{G})$. Choosing an equivalence $* \to \gT(X)$ we obtain by adjunction a map $X_*(*) \to \gT$, which extends to an $E_2$-map $f : \gE_2(X_*(*)) \to \gT$. This can be strictified to a map $\overline{f} : \widetilde{\gS}=\overline{\gE}_2(X_*(*)) \to \overline{\gT}$ of unital associative monoids in $\mathsf{Top}^\mathsf{G}$; furthermore these are cofibrant in this category by \cite[Lemma 12.7 (i)]{e2cellsI}. This gives $\overline{\gT}$ the structure of a right $\widetilde{\gS}$-module, cofibrant in $\mathsf{Top}^\mathsf{G}$, to which Theorem \ref{thm:repstab} can be applied.

The object $\gM := r_* \overline{\gT}$ satisfies
$$\gM(n) \simeq  \bigsqcup_{r(Y) = n}B\mathrm{Aut}_\mathsf{G}(Y)$$
because each $\overline{\gT}(Y)$ is contractible. If in addition
\begin{enumerate}[(I)]
\setcounter{enumi}{2}
\item\label{it:inj} the map $- \oplus X : \mathrm{Aut}_\mathsf{G}(A \oplus X^{\oplus n}) \to \mathrm{Aut}_\mathsf{G}(A \oplus X^{\oplus n+1})$ is injective for all $n \geq 0$, and

\item\label{it:canc} $Y \oplus X^{\oplus m} \cong A \oplus X^{\oplus n}$ with $1 \leq m \leq n$ implies $Y \cong A \oplus X^{\oplus n-m}$,
\end{enumerate}
 then, as explained in \cite[Section 7.3]{Krannich}, for a point $m \in B\mathrm{Aut}_\mathsf{G}(A \oplus X^{\oplus n}) \subset \gM$ the space $|W_\bullet(m)|$ is equivalent to the space $|W_n(A, X)_\bullet|$ of \cite[Definition 2.1]{RWW}, also called ``spaces of destabilisations''. The following gives a conceptual meaning to these spaces of destabilisations, analogous to that given by Theorem \ref{thm:1}.

\begin{cor}\label{cor:DestabCxIsModDec}
Under the assumptions above there is an $\mathrm{Aut}_\mathsf{G}(A \oplus X^{\oplus n})$-equivariant equivalence between the unreduced suspension of $|W_n(A, X)_\bullet|$ and $Q_\bL^{\widetilde{\gS}}(\overline{\gT})(A \oplus X^{\oplus n})$.
\end{cor}
\begin{proof}
Apply Theorem \ref{thm:repstab} to $\widetilde{\gM} = \overline{\gT}$, and use that $\overline{\gT}(Y) \simeq *$ so that the cofibre of \eqref{eq:PreKanExt} is the unreduced suspension of its fibre.
\end{proof}

\begin{rem}
In \cite[Definition 2.8]{RWW} there is formulated a simplicial complex $S_n(A,X)$ which is an ``unordereed version'' of the semi-simplicial sets $W_n(A,X)_\bullet$, and is mainly useful when the braided monoidal groupoid $(\mathsf{G}, \oplus, b, 0)$ is in fact symmetric monoidal. In this case $\gT$ has the structure of an $E_\infty$-algebra, and a similar analysis to that which we have carried out so far will show that $S_n(A,X)$ is  $\mathrm{Aut}_\mathsf{G}(A \oplus X^{\oplus n})$-equivariantly equivalent to the indecomposables $Q_\bL^{\overline{\gE}_\infty(X_*(*))}(\overline{\gT})(A \oplus X^{\oplus n})$ of $\overline{\gT}$ as a module over the free $E_\infty$-algebra on one generator. We leave the details of this argument to the appropriately motivated reader.
\end{rem}

\section{Coefficient systems, representation stability, and central stability}\label{sec:repstab}

In \cite[Section 19]{e2cellsI} it is discussed how to treat coefficient systems in the setting of Section \ref{sec:AlgFromGpd}. As a brief reminder, one fixes a commutative ring $\bk$ and works in the category $\mathsf{sMod}_\bk^\mathsf{G}$ of functors from $\mathsf{G}$ to simplicial $\bk$-modules. The constant functor $\underline{\bk}$ with value $\bk$ has the structure of a commutative algebra object in this category, and a \emph{coefficient system} $\gA$ is defined to be a right\footnote{In \cite[Section 19]{e2cellsI} left modules are considered, but there is no important difference.} $\underline{\bk}$-module. It is called \emph{discrete} if it takes values in $\bk$-modules (considered as discrete simplicial $\bk$-modules.)

Using $(-)_\bk := \bk[\mathrm{Sing}_\bullet(-)] : \mathsf{Top} \to \mathsf{sMod}_\bk$ we can transport much of the previous discussion into the category of simplicial $\bk$-modules. In particular there are unital associative monoids $\widetilde{\gS}_\bk \to \overline{\gT}_\bk$ which are cofibrant in $\mathsf{sMod}_\bk^\mathsf{G}$, and as $\overline{\gT}$ takes contractible values there is an equivalence of unital associative monoids $\overline{\gT}_\bk \overset{\sim}\to \underline{\bk}$, which is a cofibrant replacement of $\underline{\bk}$. Any coefficient system $\gA$ can therefore be considered as a right $\overline{\gT}_\bk$-module, and if $\gA^c \overset{\sim}\to \gA$ is a cofibrant replacement as such then taking Kan extensions along $r : \mathsf{G} \to \bN$ gives $\gR_\gA := r_*(\gA^c) \simeq \bL r_*(\gA)$ the structure of a right module over $\overline{\gR}_\bk := r_!(\overline{\gT}_\bk)$. By definition of homotopy Kan extension we have
$$H_{n,d}(\gR_\gA) = \bigoplus_{r(Y)=n} H_d(\mathrm{Aut}_\mathsf{G}(Y) ; \gA(Y)).$$
Using the right $\overline{\gR}_\bk$-module structure and $\sigma \in H_{1,0}(\gR_\bk)$ we can form the map $- \cdot \sigma : \gR_\gA \otimes S^{1,0} \to \gR_\gA$, and homological stability for the groups $\mathrm{Aut}_\mathsf{G}(Y)$ with coefficients in $\gA(Y)$ can be phrased as a vanishing line for the homology of the cofibre $\gR_\gA/\sigma$.

Assuming that $\gA$ is a discrete coefficient system we define
$$\mathrm{Tor}^{\underline{\bk}}_{p}(\gA, \bk)(Y) := H_{Y, d}(B(\gA, \underline{\bk}, \bk))$$
and combining \cite[Lemma 19.4]{e2cellsI} and \cite[Theorem 19.2]{e2cellsI} shows that an appropriate vanishing line for these $\mathrm{Tor}$-groups and homological stability for $\gR_\bk$ implies homological stability for $\gR_\gA$. A vanishing line for these $\mathrm{Tor}$-groups sometimes goes under the name of \emph{derived representation stability} for $\gA$. These $\mathrm{Tor}$-groups have a clear conceptual meaning: they measure how to construct $\gA$ as a cellular $\underline{\bk}$-module. (When $\mathsf{G}$ is the category of finite sets and bijections, then a $\underline{\bk}$-module recovers the notion of an $FI$-module, and $\mathrm{Tor}^{\underline{\bk}}_*(\bk, \gA)$ recovers $FI$-homology in the sense of \cite{ChurchEllenberg}.)

There is another measure of the complexity of a coefficient system $\gA$, namely the \emph{central stability homology} $\widetilde{H}_*(\gA)$ of Putman--Sam \cite{PS} and Patzt \cite{PatztCentral}. Our main goal here is to give a similar conceptual interpretation of these homology groups, and hence to revisit Patzt's theorem \cite[Theorem 5.7]{PatztCentral} relating $\widetilde{H}_*(\gA)$ and $\mathrm{Tor}^{\underline{\bk}}_*(\gA, \bk)$.

\begin{prop}\label{prop:CentralStab}
A discrete coefficient system $\gA$ may be considered as a right $\widetilde{\gS}_\bk$-module via $\widetilde{\gS}_\bk \to \overline{\gT}_\bk \overset{\sim}\to \underline{\bk}$, and then there are isomorphisms
$$H_{Y,d}^{\widetilde{\gS}_\bk}(\gA) = H_{Y, d}(B(\gA, \widetilde{\gS}_\bk, \bk)) \cong \widetilde{H}_{d-1}(\gA)_Y.$$
\end{prop}
\begin{proof}
Following Section \ref{sec:Pfthm:repstab}, the equivalence $|R_\bullet(\widetilde{\gS})| \to \widetilde{\gS}_{>0}$ and the cofibre sequence $\widetilde{\gS}_{>0} \to \widetilde{\gS} \to \bunit$ may be $\bk$-linearised, and applying $B(\gA, \widetilde{\gS}_\bk, -)$ and using that $B(\gA, \widetilde{\gS}_\bk, \widetilde{\gS}_\bk) \overset{\sim}\to \gA$ then gives a homotopy cofibre sequence
$$B(\gA, \widetilde{\gS}_\bk, |R_\bullet(\widetilde{\gS}_\bk)|) \lra \gA \lra B(\gA, \widetilde{\gS}_\bk, \bk).$$
We may commute homotopy colimits and write the left-hand term as $|B(\gA, \widetilde{\gS}_\bk, R_\bullet(\widetilde{\gS}_\bk))|$. As $R_p(\widetilde{\gS}_\bk) \simeq \widetilde{\gS}_\bk \otimes (X^{\oplus p+1})_*(\bk)$ as a left $\widetilde{\gS}_\bk$-module, we have $B(\gA, \widetilde{\gS}_\bk, R_p(\widetilde{\gS}_\bk)) \simeq \gA \otimes (X^{\oplus p+1})_*(\bk)$. The Bousfield--Kan spectral sequence for the augmented $\widetilde{\Delta}_\inj^{op}$-object $B(\gA, \widetilde{\gS}_\bk, R_\bullet(\widetilde{\gS}_\bk)) \to \gA$ therefore takes the form
$$E^1_{Y,p,q} = H_{Y,q}(\gA \otimes (X^{\oplus p+1})_*(\bk)) \Longrightarrow H_{Y, p+q+1}(B(\gA, \widetilde{\gS}_\bk, \bk))$$
for $p \geq -1$. As $\gA \otimes (X^{\oplus p+1})_*(\bk)$ is discrete this spectral sequence is supported along the line $q=0$ and so collapses at $E^2$. By definition of Day convolution it has
$$E^1_{p,0} = \colim_{\substack{(Z, f) s.t.\ \\ f: Z \oplus X^{\oplus p+1} \overset{\sim}\to Y}} \gA(Z)$$
and by definition of the $\widetilde{\Delta}_\inj^{op}$-object $R_\bullet(\widetilde{\gS}_\bk)$ the $d^1$-differential is given by the alternating sum of the maps
$$\delta_0, \delta_1, \ldots, \delta_p : \colim_{\substack{(Z, f) s.t.\ \\ f: Z \oplus X^{\oplus p+1} \overset{\sim}\to Y}} \gA(Z) \lra \colim_{\substack{(Z', f') s.t.\ \\ f': Z' \oplus X^{\oplus p} \overset{\sim}\to Y}} \gA(Z')$$
where $\delta_i$ braids the $i$th copy of $X$ in $X^{\oplus p+1}$ in front of the others to put it first, then adds it to $Z$ to form $Z' := Z \oplus X$; it then applies $\gA(Z) \to \gA(Z \oplus X)$ given by the right $\widetilde{\gS}_\bk$-module structure. Using \cite[Proposition 4.3]{PatztCentral} one finds the same description of the complex that calculates central stability homology, so that $\widetilde{H}_p(\gA)_Y \cong E^2_{p,0} \cong H_{Y, p+1}(B(\gA, \widetilde{\gS}_\bk, \bk))$ as claimed.
\end{proof}

For the following we strengthen assumption (\ref{it:inj}) of Section \ref{sec:AlgFromGpd} to 
\begin{enumerate}[(I$^\prime$)]
\setcounter{enumi}{2}
\item\label{it:ForKunneth} the map $- \oplus - : \mathrm{Aut}_\mathsf{G}(U) \times \mathrm{Aut}_\mathsf{G}(V) \to \mathrm{Aut}_\mathsf{G}(U \oplus V)$ is injective for all $U,V \in \mathsf{G}$.
\end{enumerate}
With the interpretations of $H_{*,*}^{\widetilde{\gS}_\bk}(\overline{\gT}_\bk)$ given by Corollary \ref{cor:DestabCxIsModDec} and of $H_{*,*}^{\widetilde{\gS}_\bk}(\gA)$ given by Proposition \ref{prop:CentralStab}, and the interpretation of a vanishing line for $\mathrm{Tor}^{\underline{\bk}}_{*}(\bk, \gA)$ in terms of a minimal $\underline{\bk}$-module resolution of $\gA$, the following is then a version of Patzt's \cite[Theorem 5.7]{PatztCentral}.

\begin{thm}
Let $f : \bN \to \bN$ and assume that $H_{Y,d}^{\widetilde{\gS}_\bk}(\overline{\gT}_\bk)=0$ for $d < f(r(Y))$.
\begin{enumerate}[(i)]
\item If $g : \bN \to \bN$ is such that $\mathrm{Tor}^{\underline{\bk}}_{d}(\bk, \gA)(V)=0$ for $d < g(r(V))$, then $H_{Y,d}^{\widetilde{\gS}_\bk}(\gA)=0$ for $d < \bar{g}(r(Y))$, where $\bar{g}(n) := \min\{f(p) + g(n-p) \, | \, 0 \leq p \leq n\}$.

\item If $h : \bN \to \bN$ is such that $H_{U,d}^{\widetilde{\gS}_\bk}(\gA)=0$ for $d < h(r(U))$, then $\mathrm{Tor}^{\underline{\bk}}_{d}(\bk, \gA)(Y)=0$ for $d < \bar{h}(r(Y))$, where $\bar{h}(n)$ is defined inductively by $\bar{h}(0)=h(0)$ and $\bar{h}(n) := \min\{ h(n), f(p)+\bar{h}(n-p)+1 \, | \, 1 \leq p \leq n\}$.
\end{enumerate}
\end{thm}
\begin{proof}
Consider $B(B(\bk, \widetilde{\gS}_\bk, \overline{\gT}_\bk), \overline{\gT}_\bk, \gA)$. By interchanging geometric realisations and using $B(\overline{\gT}_\bk,\overline{\gT}_\bk,\gA) \overset{\sim}\to \gA$ this is equivalent to $B(\bk, \widetilde{\gS}_\bk, \gA)$. On the other hand we may descendingly filter $B(\bk, \widetilde{\gS}_\bk, \overline{\gT}_\bk)$ by rank, as in \cite[Remark 19.5]{e2cellsI}. The associated graded is equivalent to $B(\bk, \widetilde{\gS}_\bk, \overline{\gT}_\bk)$ but its $\widetilde{\gS}_\bk$-module structure is now trivial (i.e.\ induced via the augmentation $\widetilde{\gS}_\bk \to \bk$). Thus the induced filtration of $B(B(\bk, \widetilde{\gS}_\bk, \overline{\gT}_\bk), \overline{\gT}_\bk, \gA)$ has associated graded $B(\bk, \widetilde{\gS}_\bk, \overline{\gT}_\bk) \otimes B(\bk, \overline{\gT}_\bk, \gA)$. Using (\ref{it:ForKunneth}$^\prime$) we may apply \cite[Lemma 10.6]{e2cellsI} to see that the associated spectral sequence takes the form
$$E^1_{Y, p, q} = \colim_{\substack{U \oplus V \overset{\sim}\to Y \\ r(U) = p}} H_{p+q}\left( B(\bk, \widetilde{\gS}_\bk, \overline{\gT}_\bk)(U) \otimes B(\bk, \overline{\gT}_\bk, \gA)(V) \right) \Longrightarrow H_{Y,p+q}^{\widetilde{\gS}_\bk}(\gA).$$
For such $U$ and $V$ there is also a K{\"u}nneth spectral sequence \cite[Lemma 10.5]{e2cellsI}
$$\bigoplus_{t'+t''=q} \mathrm{Tor}^\bk_s(H_{U,t'}^{\widetilde{\gS}_\bk}(\overline{\gT}_\bk), H_{V, t''}^{\overline{\gT}_\bk}(\gA)) \Longrightarrow H_{s+t}\left( B(\bk, \widetilde{\gS}_\bk, \overline{\gT}_\bk)(U) \otimes B(\bk, \overline{\gT}_\bk, \gA)(V) \right).$$
By assumption $H_{U,t'}^{\widetilde{\gS}_\bk}(\overline{\gT}_\bk)=0$ for $t' < f(p)$ as $r(U)=p$. As $\gA$ is assumed to be discrete, the discussion before \cite[Lemma 19.4]{e2cellsI} gives $H_{V, t''}^{\overline{\gT}_\bk}(\gA) \cong \mathrm{Tor}^{\underline{\bk}}_{ t''}(\bk, \gA)(V)$. 

Supposing first that $\mathrm{Tor}^{\underline{\bk}}_{d}(\bk, \gA)(V)=0$ for all $d < g(r(V))$, then the K{\"u}nneth spectral sequence implies that
$$H_{p+q}\left( B(\bk, \widetilde{\gS}_\bk, \overline{\gT}_\bk)(U) \otimes B(\bk, \overline{\gT}_\bk, \gA)(V) \right)=0 \text{ for } p+q < f(p) + g(r(Y)-p),$$
and so the first spectral sequence implies that $H_{Y,d}^{\widetilde{\gS}_\bk}(\gA)=0$ for $d < \bar{g}(r(Y))$, by definition of $\bar{g}$.

Suppose now that $H_{Y,d}^{\widetilde{\gS}_\bk}(\gA)=0$ for all $d < h(r(Y))$. Suppose for an induction that $\mathrm{Tor}^{\underline{\bk}}_{d}(\bk, \gA)(Y')=0$ for all $d < \bar{h}(r(Y'))$ and all $r(Y') < r(Y)$. The only object $U \in \mathsf{G}$ with $r(U)=0$ is $U=0$ by (\ref{it:NoRk0}), and $H_{0,*}^{\widetilde{\gS}_\bk}(\overline{\gT}_\bk)=\bk[0]$ consists of free $\bk$-modules. Thus the K{\"u}nneth spectral sequence collapses to give $E^1_{Y,0,q} = H_{Y, q}^{\overline{\gT}_\bk}(\gA)$. On the other hand if $r(U)>0$ then $r(V) < r(Y)$ and so by the inductive hypothesis $H_{V, t''}^{\overline{\gT}_\bk}(\gA)=0$ for $t'' < \bar{h}(r(V))$: it then follows by the same argument as above that for $p>0$ we have $E^1_{Y,p,q}=0$ when $p+q < f(p) + \bar{h}(r(Y)-p)$. As the differentials have the form $d^r : E^r_{Y, 0, d} \to E^r_{Y, r, d-r-1}$ the cokernel of the edge homomorphism
$$H_{Y, d}^{\widetilde{\gS}_\bk}(\gA) \lra E^1_{Y,0,d} = H_{Y, d}^{\overline{\gT}_\bk}(\gA)$$
is trivial for $d-1 < \min\{f(p) + \bar{h}(r(Y)-p) \, | \, 1 \leq p \leq r(Y)\}$. As the domain of this edge homomorphism vanishes for $d < h(r(Y))$, it follows that $H_{Y, d}^{\overline{\gT}_\bk}(\gA) \cong \mathrm{Tor}^{\underline{\bk}}_{d}(\bk, \gA)(Y)$ vanishes for $d < \bar{h}(r(Y))$.
\end{proof}

\section{The space of destabilisations and the splitting complexes}

In this section we continue to work in combinatorial setting of Section \ref{sec:AlgFromGpd}, and will explain the relationship between the connectivities of the spaces of destabilisations $|W_n(0, X)_\bullet|$ defined in \cite[Definition 2.1]{RWW}, and the connectivities of the ``$E_1$- and $E_2$-splitting complexes" $|Z^{E_1}_{\bullet}(X^{\oplus n})|$ and $|Z^{E_2}_{\bullet, \bullet}(X^{\oplus n})|$ defined in \cite[Sections 17.2 and 17.3]{e2cellsI}. 

\begin{prop}\label{prop:TransfConn}
Let $(\mathsf{G}, \oplus, b, 0)$ be a braided monoidal groupoid satisfying (\ref{it:NoRk0}), (\ref{it:TrivAut0}), (\ref{it:inj}), and (\ref{it:canc}), and suppose $r : \mathsf{G} \to \bN$ is a bijection on isomorphism classes of objects, with $X \in \mathsf{G}$ corresponding to $1 \in \bN$. Let $f : \bN \to \bN$ be a function satisfying $f(n) \leq n$ and $f(n+m) \leq f(n)+f(m)$. Then the following are equivalent:
\begin{enumerate}[(i)]
\item\label{it:TransfConn1} the homology of $|W_n(0, X)_\bullet|$ vanishes in degrees $* < f(n) - 1$ for all $n > 1$,

\item\label{it:TransfConn2} the homology of $|Z^{E_1}_{\bullet}(X^{\oplus n})|$ vanishes in degrees $* < f(n)+1$ for all $n>1$,

\item\label{it:TransfConn3} the homology of $|Z^{E_2}_{\bullet, \bullet}(X^{\oplus n})|$ vanishes in degrees $* < f(n)+2$ for all $n>1$.
\end{enumerate}
\end{prop}

Hepworth \cite[Theorem 13.2]{Hepworth} has shown that (\ref{it:TransfConn2}) $\Rightarrow$ (\ref{it:TransfConn1}) under slightly different connectivity hypotheses, see Example \ref{ex:Hepworth}.

\begin{proof}[Proof of Proposition \ref{prop:TransfConn}]
We will first show that (\ref{it:TransfConn3}) is equivalent to (\ref{it:TransfConn1}), and later the simpler statement that (\ref{it:TransfConn3}) is equivalent to (\ref{it:TransfConn2}). To do so we will use the abstract connectivity (cf.\ \cite[Definition 11.1]{e2cellsI}) $\sigma : \mathsf{G} \to [-\infty, \infty]_{\geq }$ defined by $\sigma(A) := f(r(A))$, which by assumption satisfies $\sigma \leq r$ and $\sigma * \sigma \geq \sigma$.

As before we may linearise (via $(-)_\bZ := \bZ[\mathrm{Sing}_\bullet(-)] : \mathsf{Top} \to \mathsf{sMod}_\bZ$) the $E_2$-map $f : \gE_2(X_*(*)) \to \gT$ to obtain a map $f_\bZ : \gE_2(X_*(\bZ)) \to \gT_\bZ$ of $E_2$-algebras in $\mathsf{sMod}_\bZ^\mathsf{G}$, with $\gT_\bZ(A) \simeq \bZ$ if $r(A)>0$ and $\gT_\bZ(0)=0$.  By \cite[Proposition 17.14]{e2cellsI} we have 
$$S^{0,2} \otimes Q^{E_2}_\bL(\gT_\bZ)(A)\simeq |\bZ[Z^{E_2}_{\bullet, \bullet}(A)]|.$$
As $Q^{E_2}_\bL(\gE_2(X_*(\bZ))) \simeq X_*(\bZ)$ is supported on the object $X$, we have
\begin{equation}\label{eq:AlgSide}
H_{X^{\oplus n}, d}^{E_2}(\gT_\bZ, \gE_2(X_*(\bZ))) \cong H_{X^{\oplus n}, d}^{E_2}(\gT_\bZ) \cong \widetilde{H}_{d+2}(|Z^{E_2}_{\bullet, \bullet}(X^{\oplus n})|)
\end{equation}
as long as $n > 1$. On the other hand, as $Q^{\overline{\gE}_2(X_*(\bZ))}_\bL(\overline{\gE}_2(X_*(\bZ))) \simeq 0_*(\bZ)$ is supported on the object 0, linearising the conclusion of Corollary \ref{cor:DestabCxIsModDec} gives
\begin{equation}\label{eq:ModSide}
H^{\overline{\gE}_2(X_*(\bZ))}_{X^{\oplus n}, d}(\overline{\gT}_\bZ, \overline{\gE}_2(X_*(\bZ))) \cong \widetilde{H}_{d-1}(|W_n(0, X)_\bullet|)
\end{equation}
as long as $n > 0$.

The homology groups to which (\ref{it:TransfConn3}) and (\ref{it:TransfConn1}) refer are the right-hand sides of \eqref{eq:AlgSide} and \eqref{eq:ModSide} respectively: we will compare them using the interpretations given by the left-hand sides, as relative $E_2$-algebra indecomposables and $\overline{\gE}_2(X_*(\bZ))$-module indecomposables respectively. 

Suppose first that $(\mathsf{G}, \oplus, b, 0)$ is in fact \emph{symmetric} monoidal. Then we may apply \cite[Theorem 15.9]{e2cellsI} in the category $\mathsf{C} := \mathsf{sMod}_\bZ^\mathsf{G}$ with $k=2$, because $\mathsf{G}$ is 3-monoidal (= symmetric monoidal) and so $\mathsf{C}$, with the Day convolution monoidal structure, is too. This Theorem, applied to the morphism $f_\bZ : \gE_2(X_*(\bZ)) \to \gT_\bZ$ with $\rho = r$ and with $\sigma$ an abstract connectivity such that $\sigma * \sigma \geq \sigma$ and $r \geq \sigma$, says the following: if $H^{E_2}_{X^{\oplus n}, d}(\gT_\bZ, \gE_2(X_*(\bZ)))=0$ whenever $d < \sigma(X^{\oplus n})$ then there is a morphism
\begin{equation}\label{eq:Comparison}
H^{\overline{\gE}_2(X_*(\bZ))}_{X^{\oplus n}, d}(\overline{\gT}_\bZ, \overline{\gE}_2(X_*(\bZ))) \lra H_{X^{\oplus n}, d}^{E_2}(\gT_\bZ, \gE_2(X_*(\bZ)))
\end{equation}
which is an isomorphism for $d < (\sigma * \sigma)(X^{\oplus n})$ and an epimorphism for $d < (\sigma * \sigma)(X^{\oplus n})+1$.

If $\sigma$ is such that (\ref{it:TransfConn3}) holds then by \eqref{eq:AlgSide} the assumption for the above is satisfied, and so as $\sigma * \sigma \geq \sigma$ it follows that $H^{\overline{\gE}_2(X_*(\bZ))}_{X^{\oplus n}, d}(\overline{\gT}_\bZ, \overline{\gE}_2(X_*(\bZ)))=0$ for $d < \sigma(X^{\oplus n})$, so by \eqref{eq:ModSide} it follows that the homology of $|W_n(0, X)_\bullet|$ vanishes in degrees $* < \sigma(X^{\oplus n})-1$ for $n > 1$.

In the other direction, if $\sigma$ is such that (\ref{it:TransfConn1}) holds then $H^{\overline{\gE}_2(X_*(\bZ))}_{X^{\oplus n}, d}(\overline{\gT}_\bZ, \overline{\gE}_2(X_*(\bZ)))=0$ for $d < \sigma(X^{\oplus n})$ by \eqref{eq:ModSide}. Define abstract connectivities $\sigma_k$ by
$$\sigma_k(X^{\oplus n}) := \begin{cases}
\sigma(X^{\oplus n}) & n \leq k\\
\sigma(X^{\oplus k}) & n \geq k,
\end{cases}$$
which satisfy $\sigma_k * \sigma_k \geq \sigma_k$ and $\sigma_k \leq r$. As $H_{X^{\oplus n}, 0}^{E_2}(\gT_\bZ, \gE_2(X_*(\bZ)))=0$ for all $n$ we have $H_{X^{\oplus n}, d}^{E_2}(\gT_\bZ, \gE_2(X_*(\bZ)))=0$ for $d < \sigma_1(X^{\oplus n})$, because $\sigma(X) \leq r(X)=1$ by assumption. Suppose for an induction that $H_{X^{\oplus n}, d}^{E_2}(\gT_\bZ, \gE_2(X_*(\bZ)))=0$ for $d < \sigma_k(X^{\oplus n})$. Then by \cite[Theorem 15.9]{e2cellsI} the map \eqref{eq:Comparison} is an epimorphism for $d < (\sigma_k * \sigma_k)(X^{\oplus n})+1$, and by assumption its source vanishes for $d < \sigma(X^{\oplus n})$, so we conclude that its target vanishes for
$$d < \inf(\sigma, \sigma_k * \sigma_k + 1)(X^{\oplus n}).$$
In particular it vanishes for $d < \inf(\sigma, \sigma_k + 1)(X^{\oplus n})$ and hence also for $d < \sigma_{k+1}(X^{\oplus n})$, as $\sigma_{k+1} \leq \sigma$ and $\sigma_{k+1} \leq \sigma_{k}+1$. It follows by induction that $H_{X^{\oplus n}, d}^{E_2}(\gT_\bZ, \gE_2(X_*(\bZ)))=0$ for $d < \sigma_\infty(X^{\oplus n}) = \sigma(X^{\oplus n})$. Using \eqref{eq:AlgSide} this translates to the homology of $|Z^{E_2}_{\bullet, \bullet}(X^{\oplus n})|$ vanishing in degrees $* < \sigma(X^{\oplus n})+2$ for $n>1$. This finishes the proof that (\ref{it:TransfConn3}) is equivalent to (\ref{it:TransfConn1}) if $(\mathsf{G}, \oplus, b, 0)$ is {symmetric} monoidal.

If $(\mathsf{G}, \oplus, b, 0)$ is only \emph{braided} monoidal then we can not appeal directly to \cite[Theorem 15.9]{e2cellsI}: it proof uses \cite[Theorem 15.3]{e2cellsI} which is false if $k=2$ and $\mathsf{G}$ is only braided monoidal (see Example \ref{ex:Counterexample}). However, in Appendix \ref{sec:Appendix} we show that the conclusion of \cite[Theorem 15.9]{e2cellsI} is nonetheless true when $k=2$ and $\mathsf{G}$ is only braided monoidal. Given this, the above argument goes through to show that (\ref{it:TransfConn3}) is equivalent to (\ref{it:TransfConn1}).

To see that (\ref{it:TransfConn3}) and (\ref{it:TransfConn2}) are equivalent we use the results of \cite[Section 14]{e2cellsI} for transferring vanishing lines, along with 
\begin{equation}\label{eq:AlgSideAgain}
H_{X^{\oplus n}, d}^{E_1}(\gT_\bZ) \cong \widetilde{H}_{d+1}(|Z^{E_1}_{\bullet}(X^{\oplus n})|)
\end{equation}
from \cite[Proposition 17.14]{e2cellsI}. If $\sigma$ is such that (\ref{it:TransfConn2}) holds then \eqref{eq:AlgSideAgain} shows that $H_{X^{\oplus n}, d}^{E_1}(\gT_\bZ)=0$ for $d < \sigma(X^{\oplus n})$ for $n > 1$, so letting 
$$\rho(X^{\oplus n}) := \begin{cases}
\sigma(X^{\oplus n})+1 & n > 1\\
n & n \leq 1
\end{cases}$$
we have $\rho * \rho \geq \rho$ and $H_{X^{\oplus n}, d}^{E_1}(\gT_\bZ)=0$ for $d < \rho(X^{\oplus n})-1$, so by \cite[Theorem 14.4]{e2cellsI} it follows that $H_{X^{\oplus n}, d}^{E_2}(\gT_\bZ)=0$ for $d < \rho(X^{\oplus n})-1$ (so for $d < \sigma(X^{\oplus n})$ and $n>1$), which via \eqref{eq:AlgSide} implies that (\ref{it:TransfConn3}) holds. Using the same $\rho$, \cite[Theorem 14.4]{e2cellsI} shows that (\ref{it:TransfConn3}) implies (\ref{it:TransfConn2}).
\end{proof}

\begin{example}
Let $(\mathsf{G}, \oplus, b, 0)$ be the free braided monoidal groupoid on one object $X$, so $\mathrm{Aut}_\mathsf{G}(X^{\oplus n}) \cong \beta_n$ is the braid group on $n$ strands. In this case $\gT \in \mathsf{Alg}_{E_2}(\mathsf{Top}^\mathsf{G})$ is the free $E_2$-algebra on $X_*(*)$, and so $|Z^{E_2}_{\bullet, \bullet}(A)|$ is the value at $A \in \mathsf{G}$ of the object $S^{0,2} \wedge X_*(S^0) \in \mathsf{Top}_*^{\mathsf{G}}$. This is $S^2$ when evaluated at $X$ and contractible otherwise, so in general when evaluated at $X^{\oplus n}$ its homology vanishes in degrees $* < n + 2$ for all $n > 1$. By Proposition \ref{prop:TransfConn} it then follows that $|W_n(0, X)_\bullet|$ is homologically $(n-2)$-connected. The latter may be described as an arc complex \cite[Section 5.6.2]{RWW}. Note however that we used this connectivity (and in fact that it is contractible) in the proof of Lemma \ref{lem:Damiolini}, so this is not new information.
\end{example}

\begin{example}
Similarly, if $(\mathsf{G}, \oplus, b, 0)$ be the free symmetric monoidal groupoid on one object $X$, so $\mathrm{Aut}_\mathsf{G}(X^{\oplus n}) \cong \Sigma_n$ is the $n$-th symmetric group, then $\gT$ is the free $E_\infty$-algebra on $X_*(*)$. Thus $|Z^{E_2}_{\bullet, \bullet}(A)| \simeq E_\infty(X_*(S^{2}))(A)$ by combining \cite[Theorem 13.7, 13.8, 17.4]{e2cellsI}. At $A=X^{\oplus n}$ this evaluates to $(E\Sigma_n)_+ \wedge_{\Sigma_n }(S^2)^{\wedge n}$ and so has trivial homology in degrees $* < 2n$, so in particular in degrees $* < n+2$ for all $n > 1$. By Proposition \ref{prop:TransfConn} it then follows that $|W_n(0, X)_\bullet|$ is homologically $(n-2)$-connected. The latter may be identified with the ``complex of injective words'', which gives a (very complicated) new proof for the homological high-connectivity of this semi-simplicial set.
\end{example}

\begin{example}\label{ex:Hepworth}
That $|Z^{E_1}_\bullet(X^{\oplus n})|$ be $(n-1)$-connected  is called the ``standard connectivity estimate'' in \cite[Definition 17.6]{e2cellsI}, and several examples of of braided monoidal groupoids are known to satisfy this: general linear groups over Dedekind domains of class number 1 \cite[Section 18.2]{e2cellsI}, mapping class groups of oriented surfaces \cite[Theorem 3.4]{e2cellsII}, automorphism groups of free groups \cite[Corollary 4.5]{Hepworth}. In this case Proposition \ref{prop:TransfConn} applies with $f(n) = \tfrac{n}{2}$ to show that $|W_n(0, X)_\bullet|$ has trivial homology in degrees $* < \tfrac{n-2}{2}$, i.e.\ is homologically $\tfrac{n-3}{2}$-connected. This recovers \cite[Theorem 13.2]{Hepworth} at the level of homology.
\end{example}

\begin{example}
In \cite[Section 5]{RWW} many examples are given of braided monoidal groupoids $(\mathsf{G}, \oplus, b, 0)$ such that $|W_n(0, X)_\bullet|$ is $\tfrac{n-3}{2}$-connected. For example, the groupoids corresponding to: automorphism groups of free groups \cite[Proposition 5.3 and Theorem 2.10]{RWW}, general linear groups of rings having stable rank $\leq 1$ \cite[Lemma 5.10]{RWW}, mapping class groups of orientable surfaces \cite[Lemma 5.25]{RWW} and certain 3-manifolds \cite[Section 5.7]{RWW}. Setting 
$$f(n) := \begin{cases}
\tfrac{n+1}{2} & n > 0\\
0 & n=0,
\end{cases}$$
we have $f(n) \leq n$ and $f(n+m) \leq f(n)+f(m)$, and $|W_n(0, X)_\bullet|$ has trivial homology in degrees $* < f(n)-1$. By Proposition \ref{prop:TransfConn} it then follows that for $n>1$ the space $|Z^{E_1}_{\bullet}(X^{\oplus n})|$ has trivial homology in degrees $* < \tfrac{n+3}{2}$, in all of these cases. 
\end{example}

\appendix

\section{Comparing algebra and module cells, extended}\label{sec:Appendix}

The goal of this technical appendix is to relax very slightly the hypotheses of \cite[Theorem 15.9]{e2cellsI} in the case $k=2$, as follows. (In the following $\gS$ no longer denotes the free $E_2$-algebra on one generator! The notation is parallel to \cite[Theorem 15.9]{e2cellsI}.)

\begin{thm}\label{thm:EkModuleHurewicz}
Suppose that $\mathsf{S}$ satisfies \cite[Axiom 11.19]{e2cellsI}, and that $\mathsf{G}$ is braided monoidal and Artinian. Let $\rho, \sigma \colon \mathsf{G} \to [-\infty,\infty]_{\geq}$ be abstract connectivities such that $\rho * \rho \geq \rho$, $\sigma * \sigma \geq \sigma$, and $\rho * \sigma \geq \sigma * \sigma$. If
\begin{enumerate}[\indent (i)]
\item $\gR \in \mathsf{Alg}_{E_2}(\mathsf{C})$ is such that $H^{E_2}_{g,d}(\gR)=0$ for $d < \rho(g)-1$,
\item $f \colon \gR \to \gS$ is an $E_2$-algebra map such that $H^{E_2}_{g,d}(\gS, \gR)=0$ for $d <  \sigma(g)$, and
\item $\gR$ and $\gS$ are cofibrant in $\mathsf{C}$, $0$-connective, and reduced,
\end{enumerate}
then there is a map $H^{\overline{\gR}}_{g,d}(\overline{\gS}, \overline{\gR}) \to H^{E_2}_{g,d}(\gS, \gR)$ which is an isomorphism for $d < (\sigma * \sigma)(g)$, and an epimorphism for $d < (\sigma * \sigma)(g)+1$.
\end{thm}

The only change from the $k=2$ case of \cite[Theorem 15.9]{e2cellsI} is that $\mathsf{G}$ is only required to be braided monoidal, rather than symmetric monoidal.

Let us first explain the issue. The proof of \cite[Theorem 15.9]{e2cellsI} uses \cite[Theorem 15.3]{e2cellsI}, which when $\mathsf{G}$ is symmetric monoidal provides an equivalence
\begin{equation}\label{eq:e2cells15pt3}
\overline{\gE}_2(A \vee B) \simeq \overline{\gE}_2(A) \otimes E_2^+(E_1^+(S^1 \wedge A) \otimes B)
\end{equation}
of left $\overline{\gE}_2(A)$-modules. However, if $\mathsf{G}$ is only braided monoidal then there is no such equivalence. 

To explain why, recall as in Section \ref{sec:setup} that to discuss $E_2$-algebras in a category which is only braided monoidal we use the braided version $\mathcal{C}^{\mathsf{FB}_2}_2$ of the non-unitary little 2-cubes operad \cite[Definition 12.6]{e2cellsI}, which has $\mathcal{C}^{\mathsf{FB}_2}_2(n)$ contractible for each $n>0$.

\begin{example}\label{ex:Counterexample}
Let $\mathsf{G} = \mathsf{FB}_2$, the free braided monoidal groupoid on one generator, i.e.\ $\mathsf{G} = \bigsqcup_{n \geq 0} \{n\} /\!\!/ \beta_n$, and take $\mathsf{S} = \mathsf{sMod}_\bZ$. Let $A=B=\{1\}_*(\bZ)$, with $\bZ$ considered to be in degree 0. Then on the left-hand side of \eqref{eq:e2cells15pt3} we have
$$\overline{\gE}_2(A \vee B)(\{n\}) \simeq (\bZ \oplus \bZ)^{\otimes n}.$$
This is because, by definition of Day convolution, in $\mathsf{sMod}_\bZ^\mathsf{G}$ the object $(A \vee B)^{\otimes n}$ is supported at $\{n\}$ and is here given by $\mathrm{Ind}_{\beta_1 \times \cdots \times \beta_1}^{\beta_n}((\bZ \oplus \bZ)^{\otimes n})$, so when we apply $\mathcal{C}_2^{\mathsf{FB_2}}(n) \times_{\beta_n} -$ (cf.\ \cite[Definition 12.6]{e2cellsI}) we obtain $\bZ[\mathcal{C}_2^{\mathsf{FB_2}}(n)] \otimes (\bZ \oplus \bZ)^{\otimes n} \simeq (\bZ \oplus \bZ)^{\otimes n}$, using that $\mathcal{C}_2^{\mathsf{FB_2}}(n)$ is contractible. In particular, in each grading the homology of $\overline{\gE}_2(A \vee B)$ is supported in degree zero.

On the other hand, the right-hand side of \eqref{eq:e2cells15pt3} contains as a retract $A \otimes (S^1 \wedge A) \otimes B$. This is supported on the object $\{3\}$ where it is given by
$$\mathrm{Ind}_{\beta_1 \times \beta_1 \times \beta_1}^{\beta_3}(\bZ \otimes (S^1 \wedge \bZ) \otimes \bZ),$$
which has nontrivial first homology. Thus \eqref{eq:e2cells15pt3} cannot hold.
\end{example}

Our solution to this issue will be that although \eqref{eq:e2cells15pt3} need not hold when $\mathsf{G}$ is braided monoidal, a certain connectivity estimate for the natural morphism $B \to B(\bunit, \overline{\gE}_2(A), \overline{\gE}_2(A \vee B))$ that one would deduce from \eqref{eq:e2cells15pt3} does in any case hold, and it is only this connectivity estimate that is used in the proof of \cite[Theorem 15.9]{e2cellsI}. The required connectivity estimate is as follows.

\begin{prop}
Suppose that $\mathsf{S}$ satisfies \cite[Axiom 11.19]{e2cellsI}, and that $\mathsf{G}$ is braided monoidal and Artinian. Let $\sigma, \rho : \mathsf{G} \to [-\infty, \infty]_{\geq}$ be abstract connectivities with $\sigma * \sigma \geq \sigma$, $\rho * \rho \geq \rho$, and $\rho * \sigma \geq \sigma * \sigma$. If $A \in \mathsf{C} := \mathsf{S}^\mathsf{G}$ is homologically $(\rho-1)$-connective and $B \in \mathsf{C}$ is homologically $\sigma$-connective then the natural map
$$B \lra B(\bunit, \overline{\gE}_2(A), \overline{\gE}_2(A \vee B))$$
is homologically $\sigma*\sigma$-connective.
\end{prop}
\begin{proof}
Let $\beta_{a_1, a_2, \ldots, a_r}$ denote the subgroup of the braid group $\beta_{a_1 + a_2 + \cdots + a_r}$ consisting of those braids which induce a permutation which preserves the decomposition
$$\{1,2,\ldots, a_1\} \sqcup\{a_1+1, \ldots, a_1+a_2\} \sqcup \cdots \sqcup \{a_1 + a_2 + \cdots + a_{r-1} + 1, \ldots, a_1 + a_2 + \cdots + a_r\}.$$
In the braided monoidal category $\mathsf{C}$ we have
$$(A \vee B)^{\otimes n} \cong \bigvee_{a+b = n} \mathrm{Ind}_{\beta_{a,b}}^{\beta_n} (A^{\otimes a} \otimes B^{\otimes b})$$
and so
$$\overline{\gE}_2(A \vee B) \cong \bunit \vee\bigvee_{a+ b \geq 1} \mathrm{Res}^{\beta_{a+b}}_{\beta_{a,b}}((0, \infty) \times \mathcal{C}_2^{\mathsf{FB}_2}(a+b))_+ \wedge_{\beta_{a,b}}(A^{\otimes a} \otimes B^{\otimes b}).$$
Similarly $\overline{\gE}_2(A) \cong \bunit \vee \bigvee_{n \geq 1} ((0,\infty) \times \mathcal{C}_2^{\mathsf{FB}_2}(n)_+ \wedge_{\beta_n} A^{\otimes n}$. Using these identities we may express $B(\bunit, \overline{\gE}_2(A), \overline{\gE}_2(A \vee B))$ as an analytic functor of the variables $A$ and $B$ in the form
\begin{equation}\label{eq:AnalyticFunctor}
\bigvee_{a, b \geq 0} |C(a, b)_\bullet| \wedge_{\beta_{a,b}}(A^{\otimes a} \otimes B^{\otimes b})
\end{equation}
where $C(a, b)_\bullet$ is the semi-simplicial pointed space (with free $\beta_{a,b}$-action) given as follows. The space $C(a,b)_p$ is
$$\bigvee_{\substack{a_1 + a_2 + \cdots\\ + a_p + a_{p+1} = a}} \mathrm{Ind}^{\beta_{a,b}}_{\beta_{a_1} \times\beta_{a_2} \times \cdots \times \beta_{a_p} \times \beta_{a_{p+1}, b}} \left(\left(\prod_{i=1}^p (0,\infty) \times \mathcal{C}_2^{\mathsf{FB}_2}(a_i)\right) \times (0, \infty) \times \mathcal{C}_2^{\mathsf{FB}_2}(a_{p+1}+b)\right)_+,$$
with face maps given as in the two-sided bar construction.

Under the given connectivity assumptions $A^{\otimes a} \otimes B^{\otimes b}$ is $(\rho-1)^{* a} * \sigma^{*b}$-connective, i.e.\ $(\rho^{*a} * \sigma^{* b} - a)$-connective. As $|C(a,b)_\bullet|$ is a free $\beta_{a,b}$-space, the claim will follow from the decomposition \eqref{eq:AnalyticFunctor} as long as $|C(a,b)_\bullet|$ is $a$-connective, and contractible for $b=0$. Because in this case $|C(a, b)_\bullet| \wedge_{\beta_{a,b}}(A^{\otimes a} \otimes B^{\otimes b})$ is contractible for $b=0$, and is $(\rho^{*a} * \sigma^{* b})$-connective otherwise, so is at least $\sigma * \sigma$-connective except when $(a,b) = (0,1)$.

To prove this connectivity statement we observe that the $C(a,b)_p$ are homotopy-discrete, because $(0,\infty)$ and $\mathcal{C}_2^{\mathsf{FB}_2}(n)$ are all contractible, so the semi-simplicial pointed space $C(a,b)_\bullet$ is levelwise homotopy equivalent to the semi-simplicial pointed set $\pi_0 C(a,b)_\bullet$ having
$$\pi_0 C(a,b)_p = \bigvee_{\substack{a_1 + a_2 + \cdots\\ + a_p + a_{p+1} = a}} \left(\frac{\beta_{a,b}}{\beta_{a_1} \times\beta_{a_2} \times \cdots \times \beta_{a_p} \times \beta_{a_{p+1}, b}}\right)_+.$$
This semi-simplicial pointed set admits a system of degeneracies, by setting $a_i=0$, making it a simplicial pointed set. The connectivity of this simplicial set can be analysed by the same argument as \cite[Section 4]{e2cellsII}, as we now explain.

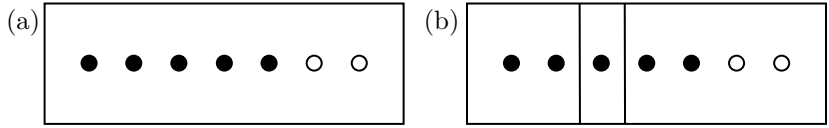
\begin{figure}[h]

\tikzset{every picture/.style={line width=0.75pt}} 

\begin{tikzpicture}[x=0.75pt,y=0.75pt,yscale=-0.75,xscale=0.75]

\draw   (29.67,10) -- (268.5,10) -- (268.5,90.75) -- (29.67,90.75) -- cycle ;
\draw  [fill={rgb, 255:red, 0; green, 0; blue, 0 }  ,fill opacity=1 ] (54.25,50) .. controls (54.25,47.24) and (56.49,45) .. (59.25,45) .. controls (62.01,45) and (64.25,47.24) .. (64.25,50) .. controls (64.25,52.76) and (62.01,55) .. (59.25,55) .. controls (56.49,55) and (54.25,52.76) .. (54.25,50) -- cycle ;
\draw   (204,50) .. controls (204,47.24) and (206.24,45) .. (209,45) .. controls (211.76,45) and (214,47.24) .. (214,50) .. controls (214,52.76) and (211.76,55) .. (209,55) .. controls (206.24,55) and (204,52.76) .. (204,50) -- cycle ;
\draw   (234,50) .. controls (234,47.24) and (236.24,45) .. (239,45) .. controls (241.76,45) and (244,47.24) .. (244,50) .. controls (244,52.76) and (241.76,55) .. (239,55) .. controls (236.24,55) and (234,52.76) .. (234,50) -- cycle ;
\draw  [fill={rgb, 255:red, 0; green, 0; blue, 0 }  ,fill opacity=1 ] (84,50) .. controls (84,47.24) and (86.24,45) .. (89,45) .. controls (91.76,45) and (94,47.24) .. (94,50) .. controls (94,52.76) and (91.76,55) .. (89,55) .. controls (86.24,55) and (84,52.76) .. (84,50) -- cycle ;
\draw  [fill={rgb, 255:red, 0; green, 0; blue, 0 }  ,fill opacity=1 ] (114,50) .. controls (114,47.24) and (116.24,45) .. (119,45) .. controls (121.76,45) and (124,47.24) .. (124,50) .. controls (124,52.76) and (121.76,55) .. (119,55) .. controls (116.24,55) and (114,52.76) .. (114,50) -- cycle ;
\draw  [fill={rgb, 255:red, 0; green, 0; blue, 0 }  ,fill opacity=1 ] (144.25,50) .. controls (144.25,47.24) and (146.49,45) .. (149.25,45) .. controls (152.01,45) and (154.25,47.24) .. (154.25,50) .. controls (154.25,52.76) and (152.01,55) .. (149.25,55) .. controls (146.49,55) and (144.25,52.76) .. (144.25,50) -- cycle ;
\draw  [fill={rgb, 255:red, 0; green, 0; blue, 0 }  ,fill opacity=1 ] (174,50) .. controls (174,47.24) and (176.24,45) .. (179,45) .. controls (181.76,45) and (184,47.24) .. (184,50) .. controls (184,52.76) and (181.76,55) .. (179,55) .. controls (176.24,55) and (174,52.76) .. (174,50) -- cycle ;
\draw   (310.67,10) -- (549.5,10) -- (549.5,90.75) -- (310.67,90.75) -- cycle ;
\draw  [fill={rgb, 255:red, 0; green, 0; blue, 0 }  ,fill opacity=1 ] (335.25,50) .. controls (335.25,47.24) and (337.49,45) .. (340.25,45) .. controls (343.01,45) and (345.25,47.24) .. (345.25,50) .. controls (345.25,52.76) and (343.01,55) .. (340.25,55) .. controls (337.49,55) and (335.25,52.76) .. (335.25,50) -- cycle ;
\draw   (485,50) .. controls (485,47.24) and (487.24,45) .. (490,45) .. controls (492.76,45) and (495,47.24) .. (495,50) .. controls (495,52.76) and (492.76,55) .. (490,55) .. controls (487.24,55) and (485,52.76) .. (485,50) -- cycle ;
\draw   (515,50) .. controls (515,47.24) and (517.24,45) .. (520,45) .. controls (522.76,45) and (525,47.24) .. (525,50) .. controls (525,52.76) and (522.76,55) .. (520,55) .. controls (517.24,55) and (515,52.76) .. (515,50) -- cycle ;
\draw  [fill={rgb, 255:red, 0; green, 0; blue, 0 }  ,fill opacity=1 ] (365,50) .. controls (365,47.24) and (367.24,45) .. (370,45) .. controls (372.76,45) and (375,47.24) .. (375,50) .. controls (375,52.76) and (372.76,55) .. (370,55) .. controls (367.24,55) and (365,52.76) .. (365,50) -- cycle ;
\draw  [fill={rgb, 255:red, 0; green, 0; blue, 0 }  ,fill opacity=1 ] (395,50) .. controls (395,47.24) and (397.24,45) .. (400,45) .. controls (402.76,45) and (405,47.24) .. (405,50) .. controls (405,52.76) and (402.76,55) .. (400,55) .. controls (397.24,55) and (395,52.76) .. (395,50) -- cycle ;
\draw  [fill={rgb, 255:red, 0; green, 0; blue, 0 }  ,fill opacity=1 ] (425.25,50) .. controls (425.25,47.24) and (427.49,45) .. (430.25,45) .. controls (433.01,45) and (435.25,47.24) .. (435.25,50) .. controls (435.25,52.76) and (433.01,55) .. (430.25,55) .. controls (427.49,55) and (425.25,52.76) .. (425.25,50) -- cycle ;
\draw  [fill={rgb, 255:red, 0; green, 0; blue, 0 }  ,fill opacity=1 ] (455,50) .. controls (455,47.24) and (457.24,45) .. (460,45) .. controls (462.76,45) and (465,47.24) .. (465,50) .. controls (465,52.76) and (462.76,55) .. (460,55) .. controls (457.24,55) and (455,52.76) .. (455,50) -- cycle ;
\draw    (385.6,9.6) -- (385.8,90.6) ;
\draw    (415.6,9.6) -- (415.8,90.6) ;

\draw (2,11.5) node [anchor=north west][inner sep=0.75pt]   [align=left] {(a)};
\draw (280,11.5) node [anchor=north west][inner sep=0.75pt]   [align=left] {(b)};

\end{tikzpicture}

	\caption{(a) The standard configuration with $a=5$ and $b=2$. (b) The standard 1-simplex $\sigma(2,1,2)$.}
	\label{fig:2}

\end{figure}

Fix a configuration of $a$ black points and $b$ white points in $[0,a+b] \times [0,1]$, as shown in Figure \ref{fig:2} (a), and let $\Sigma^{a,b}$ denote the surface given by this square with these marked points. Let the poset $\mathsf{S}(a,b)$ consist of the set of isotopy classes of smoothly embedded arcs $\alpha : [0,1] \to \Sigma^{a,b}$ disjoint from the marked points and with $\alpha(0) \in [0,a+b] \times \{0\}$ and $\alpha(1) \in [0,a+b] \times \{1\}$, such that the left-hand side of the arc contains a non-zero number of points, and the right-hand side contains all white points. Say that $[\alpha] \leq [\alpha']$ if $\alpha$ and $\alpha'$ can be represented by disjoint embedded arcs with $\alpha(0) \leq \alpha'(0) \in [0,a+b]$. Let $S(a,b)_\bullet$ denote the simplicial nerve of the poset $\mathsf{S}(a,b)$. The mapping class group of $\Sigma^{a,b}$, where diffeomorphisms must fix the boundary but are allowed to permute the black or the white marked points (but not interchange them), is the group $\beta_{a,b}$, and it acts on $S(a,b)_\bullet$. To a set of natural numbers $a_0 > 0$, $a_1, a_2, \ldots, a_{p+1} \geq 0$ there is an associated $p$-simplex $\sigma(a_0, \ldots, a_{p+1}) \in S(a,b)_p$ as shown in Figure \ref{fig:2} (b), given by vertical arcs partitioning the black points into groups of the indicated sizes. Every $p$-simplex is in the orbit of a unique $\sigma(a_0, \ldots, a_{p+1})$ (by counting the number of black points between the arcs). Furthermore, the stabiliser of this simplex under the $\beta_{a,b}$-action is the subgroup
$$\beta_{a_0} \times \cdots \times \beta_{a_{p}} \times \beta_{a_{p+1}, b}.$$
We therefore recognise the pointed simplicial set $\pi_0 C(a,b)_\bullet$ as the suspension of the simplicial set $S(a,b)_\bullet$.

When $b=0$ the poset $\mathsf{S}(a,0)$ has a maximal element, given by the arc having no points to its right, and so $\pi_0 C(a,0)_\bullet$ is indeed contractible. When $b > 0$ we must show that $|\pi_0C(a,b)_\bullet|$ is $a$-connective, i.e.\ that $|S(a,b)_\bullet|$ is $(a-2)$-connected. We will do this by induction on $a$, using the Nerve Theorem as formulated in \cite[Corollary 4.2]{e2cellsII}. It clearly holds for $a \leq 1$.

Let $A(\Sigma^{a,b})$ denote the simplicial complex with vertices the isotopy classes of smoothly embedded arcs $\gamma : [0,1] \to \Sigma^{a,b}$ with $\gamma(0)= (0,\tfrac{1}{2})$ and $\gamma(1)$ a black marked point. A collection $[\gamma_0], \ldots, [\gamma_p]$ spans a simplex if the $\gamma_i$ can be realised disjointly except at $\gamma_i(0)= (0,\tfrac{1}{2})$. By a theorem of Hatcher and Wahl \cite[Proposition 7.2]{HatcherWahl} the simplicial complex $A(\Sigma^{a,b})$ is $(a-2)$-connected. We consider the functor
$$F : \mathsf{Simp}(A(\Sigma^{a,b}))^{op} \lra \{\text{downwards-closed subposets of }\mathsf{S}(a,b)^{op}\}$$
which assigns to a simplex $\{[\gamma_0], \ldots, [\gamma_p]\}$ of $A(\Sigma^{a,b})$ the subposet of $\mathsf{S}(a,b)^{op}$ given by those $[\alpha]$'s such that the arcs $\alpha, \gamma_0, \ldots, \gamma_p$ can be realised disjointly: see Figure \ref{fig:3} (a). This is clearly a downwards-closed subposet, and defines a functor.

\begin{figure}[h]

\tikzset{every picture/.style={line width=0.75pt}} 

\begin{tikzpicture}[x=0.75pt,y=0.75pt,yscale=-0.75,xscale=0.75]
\draw   (29.67,10) -- (268.5,10) -- (268.5,90.75) -- (29.67,90.75) -- cycle ;
\draw  [fill={rgb, 255:red, 0; green, 0; blue, 0 }  ,fill opacity=1 ] (54.25,50) .. controls (54.25,47.24) and (56.49,45) .. (59.25,45) .. controls (62.01,45) and (64.25,47.24) .. (64.25,50) .. controls (64.25,52.76) and (62.01,55) .. (59.25,55) .. controls (56.49,55) and (54.25,52.76) .. (54.25,50) -- cycle ;
\draw   (204,50) .. controls (204,47.24) and (206.24,45) .. (209,45) .. controls (211.76,45) and (214,47.24) .. (214,50) .. controls (214,52.76) and (211.76,55) .. (209,55) .. controls (206.24,55) and (204,52.76) .. (204,50) -- cycle ;
\draw   (234,50) .. controls (234,47.24) and (236.24,45) .. (239,45) .. controls (241.76,45) and (244,47.24) .. (244,50) .. controls (244,52.76) and (241.76,55) .. (239,55) .. controls (236.24,55) and (234,52.76) .. (234,50) -- cycle ;
\draw  [fill={rgb, 255:red, 0; green, 0; blue, 0 }  ,fill opacity=1 ] (84,50) .. controls (84,47.24) and (86.24,45) .. (89,45) .. controls (91.76,45) and (94,47.24) .. (94,50) .. controls (94,52.76) and (91.76,55) .. (89,55) .. controls (86.24,55) and (84,52.76) .. (84,50) -- cycle ;
\draw  [fill={rgb, 255:red, 0; green, 0; blue, 0 }  ,fill opacity=1 ] (114,50) .. controls (114,47.24) and (116.24,45) .. (119,45) .. controls (121.76,45) and (124,47.24) .. (124,50) .. controls (124,52.76) and (121.76,55) .. (119,55) .. controls (116.24,55) and (114,52.76) .. (114,50) -- cycle ;
\draw  [fill={rgb, 255:red, 0; green, 0; blue, 0 }  ,fill opacity=1 ] (144.25,50) .. controls (144.25,47.24) and (146.49,45) .. (149.25,45) .. controls (152.01,45) and (154.25,47.24) .. (154.25,50) .. controls (154.25,52.76) and (152.01,55) .. (149.25,55) .. controls (146.49,55) and (144.25,52.76) .. (144.25,50) -- cycle ;
\draw  [fill={rgb, 255:red, 0; green, 0; blue, 0 }  ,fill opacity=1 ] (174,50) .. controls (174,47.24) and (176.24,45) .. (179,45) .. controls (181.76,45) and (184,47.24) .. (184,50) .. controls (184,52.76) and (181.76,55) .. (179,55) .. controls (176.24,55) and (174,52.76) .. (174,50) -- cycle ;
\draw   (310.67,10) -- (549.5,10) -- (549.5,90.75) -- (310.67,90.75) -- cycle ;
\draw  [fill={rgb, 255:red, 0; green, 0; blue, 0 }  ,fill opacity=1 ] (335.25,50) .. controls (335.25,47.24) and (337.49,45) .. (340.25,45) .. controls (343.01,45) and (345.25,47.24) .. (345.25,50) .. controls (345.25,52.76) and (343.01,55) .. (340.25,55) .. controls (337.49,55) and (335.25,52.76) .. (335.25,50) -- cycle ;
\draw   (485,50) .. controls (485,47.24) and (487.24,45) .. (490,45) .. controls (492.76,45) and (495,47.24) .. (495,50) .. controls (495,52.76) and (492.76,55) .. (490,55) .. controls (487.24,55) and (485,52.76) .. (485,50) -- cycle ;
\draw   (515,50) .. controls (515,47.24) and (517.24,45) .. (520,45) .. controls (522.76,45) and (525,47.24) .. (525,50) .. controls (525,52.76) and (522.76,55) .. (520,55) .. controls (517.24,55) and (515,52.76) .. (515,50) -- cycle ;
\draw  [fill={rgb, 255:red, 0; green, 0; blue, 0 }  ,fill opacity=1 ] (365,50) .. controls (365,47.24) and (367.24,45) .. (370,45) .. controls (372.76,45) and (375,47.24) .. (375,50) .. controls (375,52.76) and (372.76,55) .. (370,55) .. controls (367.24,55) and (365,52.76) .. (365,50) -- cycle ;
\draw  [fill={rgb, 255:red, 0; green, 0; blue, 0 }  ,fill opacity=1 ] (395,50) .. controls (395,47.24) and (397.24,45) .. (400,45) .. controls (402.76,45) and (405,47.24) .. (405,50) .. controls (405,52.76) and (402.76,55) .. (400,55) .. controls (397.24,55) and (395,52.76) .. (395,50) -- cycle ;
\draw  [fill={rgb, 255:red, 0; green, 0; blue, 0 }  ,fill opacity=1 ] (425.25,50) .. controls (425.25,47.24) and (427.49,45) .. (430.25,45) .. controls (433.01,45) and (435.25,47.24) .. (435.25,50) .. controls (435.25,52.76) and (433.01,55) .. (430.25,55) .. controls (427.49,55) and (425.25,52.76) .. (425.25,50) -- cycle ;
\draw  [fill={rgb, 255:red, 0; green, 0; blue, 0 }  ,fill opacity=1 ] (455,50) .. controls (455,47.24) and (457.24,45) .. (460,45) .. controls (462.76,45) and (465,47.24) .. (465,50) .. controls (465,52.76) and (462.76,55) .. (460,55) .. controls (457.24,55) and (455,52.76) .. (455,50) -- cycle ;
\draw  [dash pattern={on 0.84pt off 2.51pt}]  (30,50) .. controls (49.08,35.69) and (94.33,24.17) .. (99.67,37.17) .. controls (105,50.17) and (100.72,63.41) .. (112.33,65.5) .. controls (123.95,67.59) and (132.67,63.83) .. (149.08,50.38) ;
\draw  [dash pattern={on 0.84pt off 2.51pt}]  (30,50) .. controls (50.67,84.5) and (71,69.17) .. (89,50) ;
\draw    (140,91.17) .. controls (139.95,88.08) and (112.72,85.48) .. (101.33,76.17) .. controls (89.95,66.86) and (96,56.17) .. (97,50) .. controls (98,43.83) and (85.67,37.83) .. (81,44.5) .. controls (76.33,51.17) and (76,58.5) .. (60.33,61.17) .. controls (44.67,63.83) and (42.33,50.33) .. (48,46) .. controls (53.67,41.67) and (92,33.5) .. (97.67,42.17) .. controls (103.33,50.83) and (96,59.17) .. (99.33,66.83) .. controls (102.67,74.5) and (120.67,84.83) .. (138.33,83.17) .. controls (156,81.5) and (170.33,48.17) .. (170,10.17) ;
\draw  [dash pattern={on 0.84pt off 2.51pt}]  (311,48.67) .. controls (330.08,34.35) and (375.33,22.83) .. (380.67,35.83) .. controls (386,48.83) and (381.72,62.08) .. (393.33,64.17) .. controls (404.95,66.25) and (413.67,62.5) .. (430.08,49.04) ;
\draw  [dash pattern={on 0.84pt off 2.51pt}]  (311,48.67) .. controls (331.67,83.17) and (352,67.83) .. (370,48.67) ;
\draw    (313,91.17) .. controls (312.67,81.83) and (312.33,71.83) .. (318.33,67.83) .. controls (324.33,63.83) and (323,75.17) .. (344.67,71.17) .. controls (366.33,67.17) and (376.67,56.17) .. (377.67,50) .. controls (378.67,43.83) and (368.67,38.83) .. (364,45.5) .. controls (359.33,52.17) and (355.33,62.17) .. (339.67,64.83) .. controls (324,67.5) and (318,50.33) .. (323.67,46) .. controls (329.33,41.67) and (370.33,27.5) .. (376,36.17) .. controls (381.67,44.83) and (380,58.17) .. (388.33,66.5) .. controls (396.67,74.83) and (399,74.17) .. (420,65.5) .. controls (441,56.83) and (442.33,46.17) .. (431.67,42.17) .. controls (421,38.17) and (411,62.5) .. (397,60.5) .. controls (383,58.5) and (394.33,43.17) .. (378.33,28.83) .. controls (362.33,14.5) and (323,41.17) .. (317.33,38.5) .. controls (311.67,35.83) and (313,24.17) .. (313,10.5) ;

\draw (3,11.5) node [anchor=north west][inner sep=0.75pt]   [align=left] {(a)};
\draw (284,11.5) node [anchor=north west][inner sep=0.75pt]   [align=left] {(b)};
\draw (38.67,18) node [anchor=north west][inner sep=0.75pt]    {$\gamma _{0}$};
\draw (37.33,71) node [anchor=north west][inner sep=0.75pt]    {$\gamma _{1}$};
\draw (174,13.4) node [anchor=north west][inner sep=0.75pt]    {$\alpha $};

\end{tikzpicture}

\caption{(a) Dotted arcs $\gamma_0$ and $\gamma_1$, and a solid arc $\alpha$ in $F(\{[\gamma_0], [\gamma_1]\})$. (b) The maximal element of $F(\{[\gamma_0], [\gamma_1]\})$.}
\label{fig:3}
\end{figure}
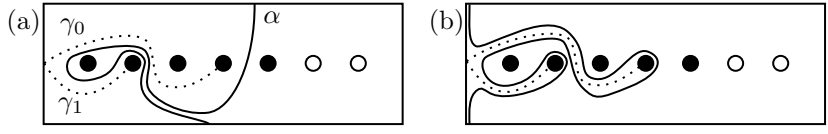

We apply the Nerve Theorem \cite[Corollary 4.2]{e2cellsII} to this functor, with the choices $t_{\mathsf{Simp}(A(\Sigma^{a,b}))}([\gamma_0], \ldots, [\gamma_p]) := p$, $t_{\mathsf{S}(a,b)^{op}}([\alpha]) := \#\{\text{black points to the right of $\alpha$}\}$, and $n:= a-1$. We verify the hypotheses of this theorem:
\begin{enumerate}[(i)]
\item $\mathsf{Simp}(A(\Sigma^{a,b}))$ is $(a-2)$-connected, by Hatcher and Wahl's theorem.

\item $\mathsf{Simp}(A(\Sigma^{a,b}))_{< \{[\gamma_0], \ldots, [\gamma_p]\}}$ is the poset of simplices of the boundary of $\Delta^p$, so is $(p-2)$-connected. The subposet $F([\gamma_0], \ldots, [\gamma_p]) \subset \mathsf{S}(a,b)^{op}$ has a maximal element, given by an arc which has \emph{precisely} the points $\gamma_0(1), \ldots, \gamma_p(1)$ to its left nd runs parallel to the $\gamma_i$ as in Figure \ref{fig:3} (b), so is contractible.

\item $(\mathsf{S}(a,b)^{op})_{<[\alpha]} = (\mathsf{S}(a,b)_{> [\alpha]})^{op}$, and if $\alpha$ has $k$ black points to its right then $\mathsf{S}(a,b)_{> [\alpha]} \cong \mathsf{S}(k,b)$, which by induction may be supposed to be $(k-2)$-connected (i.e.\ $(t_{\mathsf{S}(a,b)^{op}}([\alpha])-2)$-connected). The subposet $\mathsf{Simp}(A(\Sigma^{a,b}))_{[\alpha]}$ may be identified with $\mathsf{Simp}(A(\Sigma^{a-k,0}))$ which by Hatcher and Wahl's theorem is $(a-k-2)$-connected (i.e.\ $((a-1) - t_{\mathsf{S}(a,b)^{op}}([\alpha]) - 1)$-connected).
\end{enumerate}
It follows from the Nerve Theorem that $\mathsf{S}(a,b)^{op}$ is $(a-2)$-connected, as required.
\end{proof}

\bibliographystyle{plain}
\bibliography{MainBib}  

\end{document}